\documentclass{article}
\usepackage[utf8]{inputenc}

\usepackage{amssymb,amsmath,amsthm}
    \usepackage{mathtools}  
    \numberwithin{equation}{section}    
\usepackage{enumitem}   
\usepackage[all,cmtip]{xy}
\usepackage{tikz-cd}
    \usetikzlibrary{calc}   
    \usetikzlibrary{backgrounds}    
\usepackage{contour}    
    \contourlength{1pt}
\usepackage{hyperref}   
    \hypersetup{
        linktoc=all
    }
\usepackage{a4wide} 
\tikzset{curve/.style={settings={#1},to path={(\tikztostart)
    .. controls ($(\tikztostart)!\pv{pos}!(\tikztotarget)!\pv{height}!270:(\tikztotarget)$)
    and ($(\tikztostart)!1-\pv{pos}!(\tikztotarget)!\pv{height}!270:(\tikztotarget)$)
    .. (\tikztotarget)\tikztonodes}},
    settings/.code={\tikzset{quiver/.cd,#1}
        \def\pv##1{\pgfkeysvalueof{/tikz/quiver/##1}}},
    quiver/.cd,pos/.initial=0.35,height/.initial=0}

\tikzset{tail reversed/.code={\pgfsetarrowsstart{tikzcd to}}}
\tikzset{2tail/.code={\pgfsetarrowsstart{Implies[reversed]}}}
\tikzset{2tail reversed/.code={\pgfsetarrowsstart{Implies}}}
\tikzset{no body/.style={/tikz/dash pattern=on 0 off 1mm}}

\theoremstyle{definition}
\begingroup
    \newtheorem{remark}[equation]{Remark}
    \newtheorem{warning}[equation]{Warning}

    \newtheorem{definition}[equation]{Definition}
\endgroup
\theoremstyle{plain}
\begingroup
    \newtheorem{theorem}[equation]{Theorem}
    \newtheorem{lemma}[equation]{Lemma}
    \newtheorem{proposition}[equation]{Proposition}
    \newtheorem{corollary}[equation]{Corollary}
    \newtheorem{exercise}[equation]{Exercise}
\endgroup
\newtheorem*{thmintro}{Theorem \ref{thm:tricolimit}}
\newtheorem*{corointro}{Corollary \ref{coro:colimit-cat}}
\newtheorem*{lemmaintro}{Lemma \ref{lem:1-fibrationsliftfractions}}
\newtheorem*{thmintro2}{Theorem \ref{th:exactness}}

\newcommand{\B}{\mathcal{B}}

\newcommand{\E}{\mathcal{E}}
\newcommand{\eps}{\varepsilon}

\newcommand{\cc}[1]{\mathcal{#1}}
\newcommand{\Set}{\textnormal{\textbf{Set}}}
\newcommand{\Cat}{\textnormal{\textbf{Cat}}}
\newcommand{\Bicat}{\textnormal{\textbf{Bicat}}}

\DeclareMathOperator{\el}{{\sf el}} 
\DeclareMathOperator{\coFib}{coFib} 
\newcommand{\yo}{\text{\usefont{U}{min}{m}{n}\symbol{'210}}}
    \DeclareFontFamily{U}{min}{}
    \DeclareFontShape{U}{min}{m}{n}{<-> udmj30}{}
    
\newcommand{\mr}[1]{\stackrel{#1}{\longrightarrow}}
\newcommand{\ml}[1]{\stackrel{#1}{\longleftarrow}}
\newcommand{\Mr}[1]{\stackrel{#1}{\Rightarrow}}
\newcommand{\Mreq}[1]{\stackrel{#1}{\simeq}}
\newcommand{\xr}[1]{\xrightarrow{#1}}

\makeatletter
\newcounter{axiom}
\newcommand{\axiom}[2][\relax]{%
    \if#1\relax{\textnormal{\textbf{#2}}%
            \stepcounter{axiom}%
            \def\@currentlabel{\textnormal{\textbf{#2}}}%
            \def\@currentHref{axiom.\theaxiom}%
            \label{#2}%
            \Hy@raisedlink{\hyper@anchorstart{\@currentHref}}%
            \Hy@raisedlink{\hyper@anchorend}}
    \else{\textnormal{\textbf{#2}}%
            \stepcounter{axiom}%
            \def\@currentlabel{\textnormal{\textbf{#1}}}%
            \def\@currentHref{axiom.\theaxiom}%
            \label{#1}%
            \Hy@raisedlink{\hyper@anchorstart{\@currentHref}}%
            \Hy@raisedlink{\hyper@anchorend}}%
    \fi
}
\newcommand{\itemaxiom}[1][\relax]{%
    \if#1\relax\item\else\item[\axiom{#1}]\fi
}
\newcommand{\refaxioms}[2]{\hyperref[#1]{\textnormal{\textbf{#2}}}}

\def\varcolim@#1#2{%
    \vtop{\m@th\ialign{##\cr
        \hfil$#1\operator@font colim$\hfil\cr
        \noalign{\nointerlineskip\kern1.5\ex@}#2\cr
        \noalign{\nointerlineskip\kern-\ex@}\cr}}
}
\def\Lim{%
    \mathop{\mathpalette\varlim@{\leftarrowfill@\textstyle}}\nmlimits@
}
\def\Colim{%
    \mathop{\mathpalette\varcolim@{\rightarrowfill@\textstyle}}\nmlimits@
}

\newcommand*{\c@rc}{%
    \mathchoice%
        {\raise0.073em\hbox{$\scriptstyle\mathchar"020E$}}%
        {\raise0.073em\hbox{$\scriptstyle\mathchar"020E$}}%
        {\raise0.054em\hbox{$\scriptscriptstyle\mathchar"020E$}}%
        {\raise0em\hbox{$\scriptscriptstyle\mathchar"020E$}}%
}
\newcommand*{\@wto}[2]{%
        \sbox0{$\m@th#1\rightarrow$}%
        \sbox1{$\m@th#1\c@rc$}%
        \dimen@=\wd0%
        \advance\dimen@ by -\wd1 %
        \divide\dimen@ by 2 %
        \rlap{%
            \hskip\dimen@\hbox{$\m@th#1\c@rc$}
        }%
        \copy0
    }
\newcommand{\wto}{\mathrel{\mathpalette\@wto{}}}

\makeatother

\title{The Three \emph{F}'s for Bicategories I:\\ Localization by Fractions is Exact}
\author{Bustillo Vazquez P., Pronk D., Szyld M.}
\date{}

\begin{document}
\maketitle

\begin{abstract}
We study the interaction between the notions of filteredness, fractions and fibrations in the theory of bicategories, generalizing classical results for categories. 
We give an explicit formula for filtered pseudo-colimits of categories indexed by a bicategory, and we use it to compute the hom-categories of a bicategory of fractions.
As a consequence, we show that the canonical pseudo-functor into a bicategory of fractions is exact.
\end{abstract}

\section*{Introduction}
The rich interaction between the three $F$ notions of {\em filtered} categories, categories of {\em fractions} and Grothendieck {\em fibrations} goes back to \cite{GZ67} and to \emph{SGA4} \cite[Exp. VI, \S 6]{GSV72}, where the following results are shown. 
\begin{enumerate}
    \itemaxiom[(A)]
    For any pseudo-functor $F\colon \mathbf{C}^{\textnormal{op}} \to \Cat$ from a category $\mathbf{C}$, its pseudo-colimit is given by the localization at the Cartesian arrows of the fibration determined by its Grothendieck construction $\el F$.
    
    \itemaxiom[(B)]
    For any fibration $P\colon \mathbf{E} \to \mathbf{B}$, if $\mathbf{B}^{\textnormal{op}}$ is pseudofiltered, then its family of Cartesian arrows satisfies the axioms for a calculus of fractions.
    
    \itemaxiom[(C)]
    The hom-sets of a category of right fractions $\mathbf{C}[\cc{W}^{-1}]$ are filtered colimits of hom-sets of $\mathbf{C}$, and the localization functor $\mathbf{C} \to \mathbf{C}[\cc{W}^{-1}]$ is left exact.
\end{enumerate}
Throughout this introduction, we will refer to these results as \ref{(A)}, \ref{(B)}, and \ref{(C)}.

Since then, these three $F$ notions have been generalized to 2-dimensional category theory:
\begin{enumerate}
    \item
    In \cite{Ken92} a set of axioms for \emph{filtered 2-categories} is given, and the same axioms also make sense for bicategories. We define \emph{pseudofiltered bicategories} in Section \ref{sec:filtered}.

    \item
    In \cite{Pro96} the theory of \emph{bicategories of fractions} is developed, generalizing the axioms and the construction in \cite{GZ67}. 

    \item
    \emph{Fibrations of 2-categories and bicategories} are introduced in \cite{Her99,Bak,Buc14}. 
\end{enumerate}

\begin{center}
{\em
    In this paper we show that the results \ref{(A)}, \ref{(B)}, and \ref{(C)} also hold for these three notions.
}
\end{center}

The authors of \cite{Buc14,Bak} generalize the Grothendieck construction $\el F$ to trihomomorphisms $F$ from a bicategory into the tricategory $\Bicat$ of bicategories. Result \ref{(A)} becomes in this context:
\begin{thmintro} 
For any trihomomorphism  $F\colon \mathcal{B} \to \Bicat$ from a bicategory $\mathcal{B}$, its tricolimit in $\Bicat$ is given by the localization of its Grothendieck construction $\el F$ at both the Cartesian arrows and the Cartesian 2-cells.
\end{thmintro}
\noindent We show how, with the proper setup, the proof of this theorem can be given as a one-line computation using the fact, shown in \cite[Prop 3.3.12]{Buc14}, that the Grothendieck construction provides an equivalence between trihomomorphisms into $\Bicat$ and fibrations between bicategories.
Note that Theorem \ref{thm:tricolimit} involves the localization of a bicategory at both arrows and 2-cells, which we define but won't use at this level of generality (in \cite{BPS}, the sequel to this paper, we will compute this tricolimit for the case when $\mathcal{B}$ is filtered).
The less general situation that is relevant to us here is that of a pseudo-functor $F\colon \mathcal{B} \to \Cat$ from a bicategory $\mathcal{B}$, for which we show that it suffices to localize only at the Cartesian arrows (in what follows $\pi_0$ denotes the left adjoint to the inclusion $d$ of $\Cat$ into $\Bicat$):
\begin{corointro}
For any pseudo-functor $F\colon \mathcal{B} \to \Cat$ from a bicategory $\mathcal{B}$, its pseudo-colimit is given by the localization of $\pi_0 (\el d F)$ at the family of equivalence classes of Cartesian arrows.
\end{corointro}
\noindent This result generalizes \ref{(A)} to the case in which $\cc{C}$ is a bicategory.

We also show (see Corollary \ref{coro:BpseudofilteredEfractionsasinSGA}) that \ref{(B)} still holds when each of the three $F$ notions is replaced by their bicategorical analogues introduced in items 1, 2, and 3 above.  
This follows from the following more general Lifting Fractions Lemma (since in a co-pseudofiltered bicategory, the collection of all arrows satisfies right fractions).
\begin{lemmaintro}[Lifting Fractions Lemma]
For a fibration of bicategories $\cc{E} \to \cc{B}$, if a family of arrows in $\cc{B}$ satisfies right fractions, then so does the family of Cartesian arrows over it.
\end{lemmaintro}
\noindent We prove this lemma using a new set \refaxioms{0-Frc}{Frc} of fractions axioms for a family of arrows of a bicategory that we introduce. This set of axioms is simpler but equivalent to the original one in \cite{Pro96}. 
Working with the new set of axioms instead of the original one makes a significant difference here.
This proof provided the main motivation for the new formulation of the axioms.

Since applying $\pi_0$ to a bicategorical calculus of fractions yields a 1-dimensional calculus of fractions (see Lemma \ref{lem:pi0}), it follows that if $\mathcal{B}$ is pseudofiltered then the localization in Corollary \ref{coro:colimit-cat} can be computed as a category of fractions. 
This yields an explicit formula for filtered pseudo-colimits of categories indexed by a bicategory in terms of the data of $\cc{B}$ and the image of $F$.
We give the details in Proposition \ref{prop:filteredcolimit}.
We note that a similar formula is given in \cite{DS06} for the case in which $\cc{B}$ is a (strict) 2-category, and $F$ is a (strict) 2-functor; see also \cite{Dat14} for the relation with the bicategory of fractions.
Our formula in Proposition \ref{prop:filteredcolimit} is the generalization of the formula in \cite{DS06} to the non-strict case, and can immediately be seen to match theirs when the coherence datum is dropped. 
We sketch now how we apply this formula in  Section \ref{sec:homcat} to establish result \ref{(C)} for bicategories of fractions, where  both $\cc{B} \coloneqq (\mathcal{W}/A)^{\textnormal{op}}$ and $F \coloneqq F_A^B$ are non-strict.

Recall from \cite{GZ67} that the hom-sets of the category of fractions are given as filtered colimits of sets.
More precisely, for a family $\cc{W}$ of arrows of a category $\cc{C}$, and a pair of objects $A,B$, we can define a slice category $\mathcal{W}/A$, whose objects are the arrows in $\cc{W}$ with codomain $A$ and whose arrows are given by commutative triangles, and a functor
\begin{equation}\label{eq:Fab}
    F_A^B\colon
        (\mathcal{W}/A)^{\textnormal{op}}
        \mr{U}
        \cc{C}^{\textnormal{op}}
        \xr{\cc{C}(-,B)}
        \Set
\end{equation}
where $U$ is the functor that maps each arrow to its domain.
When $\cc{W}$ satisfies fractions, this is a filtered diagram of sets whose colimit is $\cc{C}[\cc{W}^{-1}](A,B)$.
We show in Section \ref{subsec:generalizingGZ} that a similar description exists for the hom-categories of a bicategory of fractions.
For a family $\cc{W}$ of arrows of a bicategory $\cc{B}$, and a pair of objects $A,B$, we define a slice bicategory $\mathcal{W}/A$ and a pseudo-functor
\begin{equation*}
    F_A^B\colon
        (\mathcal{W}/A)^{\textnormal{op}}
        \mr{U}
        \cc{B}^{\textnormal{op}}
        \xr{\cc{B}(-,B)}
        \Cat
\end{equation*}
similarly to \eqref{eq:Fab}.
As a further application of the {Lifting Fractions} Lemma \ref{lem:1-fibrationsliftfractions} to the fibration given by $U$, we show that if $\cc{W}$ satisfies fractions then $(\mathcal{W}/A)^{\textnormal{op}}$ is a filtered bicategory. Furthermore, we show that the formula for the pseudo-colimit of $F_A^B$, which one gets from Proposition \ref{prop:filteredcolimit}, 
can be seen to match the construction in \cite{Pro96} of the hom-categories $\cc{B}[\cc{W}^{-1}](A,B)$.
This is the content of Proposition \ref{prop:homsascolim}, establishing the first part of the result \ref{(C)} for bicategories of fractions.

Since composition of arrows in a 1-category of fractions does not depend on the choice of the Ore squares used in the construction, we obtain as a first consequence of Proposition \ref{prop:homsascolim} that the vertical composition of 2-cells in \cite[p.258]{Pro96} does not depend on the choice of Ore squares in $\cc{B}$.
The fact that the vertical composition of 2-cells is independent of these choices was also proved directly in \cite[Prop. 5.1]{Tom16}.

As another application of Proposition \ref{prop:homsascolim}, we show in Section \ref{subsec:exactness} how this result can be combined with an exactness property of $\Cat$ to show the second part of result \ref{(C)} for bicategories.
In \cite{Can16}, see also \cite{DDS18}, the commutativity in $\Cat$ of filtered pseudo-colimits and finite weighted limits is shown, when these limits are indexed by strict 2-functors with 2-categories as domain.
We use the facts that any bicategory is biequivalent to a 2-category and that any $\Cat$-valued pseudo-functor from a 2-category is equivalent to a 2-functor to deduce that this commutativity still holds for pseudo-functors from bicategories.
Finally, following the same method of the 1-dimensional original proof in \cite[\S I.3]{GZ67}, we combine this commutativity result with Proposition \ref{prop:homsascolim} to prove the following theorem, that finishes to establish \ref{(C)} for bicategories:
\begin{thmintro2}
Let $\mathcal{B}$ be a bicategory and $\mathcal{W}$ be a right bicalculus of fractions.
Then the localization pseudo-functor $\mathcal{B} \to \mathcal{B}[\mathcal{W}^{-1}]$ commutes with finite weighted bilimits. 
\end{thmintro2}

\subsection*{Organization}
The paper is organized as follows.
In Sections \ref{subsec:notations} and \ref{subsec:facts}, we introduce the notation we will be using and recall basic properties of bicategory theory.
In Section \ref{subsec:fibrations}, we introduce some classical properties of fibrations of bicategories that we will need in the paper.

In Section \ref{sec:filtered}, we introduce a new set of axioms for a pseudofiltered bicategory, that can be related to the filtered axioms for a 2-category in \cite{Ken92} just like the axioms for a pseudofiltered category are related to the axioms for a filtered category in \emph{SGA 4} \cite{GSV72} (see Propositions \ref{prop:connected} and \ref{prop:pflt-axioms}).
We show how these axioms are equivalent to (a subset of) the ones introduced in \cite{DS06} and corrected in \cite{DS21}.

In Section \ref{sec:diagrammatic}, we introduce a new set \refaxioms{0-Frc}{Frc} of fractions axioms for a family $\cc{W}$ of arrows of a bicategory and we show that these axioms are equivalent to the original ones in \cite{Pro96} (and to the modified version that can be found in \cite{PS21}).
Using these axioms, we show the {Lifting Fractions} Lemma \ref{lem:1-fibrationsliftfractions} mentioned in the introduction, and its Corollary \ref{coro:BpseudofilteredEfractionsasinSGA} that establishes the analogue of result \ref{(B)} for bicategories.

In Section \ref{sec:colimits}, Theorem \ref{thm:tricolimit}, we generalize result \ref{(A)} to an arbitrary trihomomorphism $F$ into $\Bicat$.
We then show in Corollary \ref{coro:colimit-cat} how, when $F$ takes its values in $\Cat$, we obtain a simpler formula for its pseudo-colimit as a localization of a 1-category, which can be computed by fractions when the indexing bicategory is pseudofiltered.
We give the explicit formula of this pseudo-colimit in Proposition \ref{prop:filteredcolimit}.

Finally, in Section \ref{sec:homcat} we show result \ref{(C)} for bicategories, as detailed above.

\noindent{\bf Acknowledgements.}  
The second author acknowledges the support of the Natural Sciences and Engineering Research Council
of Canada (NSERC) through a Discovery Grant.
The third author acknowledges the support of the NSERC and the Atlantic Association for Research in the Mathematical Sciences (AARMS).

We wish to thank Eduardo J. Dubuc for several helpful conversations on the subject.

\section{Preliminaries}\label{sec:preliminaries}

\subsection{Notation}\label{subsec:notations}
Throughout the paper, in addition to the usual \texttt{hyperref} links, we also use links to tie together the different axioms that are introduced: clicking on the name of an axiom (that we denote with boldface letters) will take the reader to its definition.

\begin{enumerate}[align=left, leftmargin=*]
    \item
    Abstract categories are written using a bold font style ($\mathbf{A}, \mathbf{B}, \mathbf{C}$), and  bicategories using a calligraphic font style ($\mathcal{A}$, $\mathcal{B}$, $\mathcal{C}$).
    $\Set$ denotes the category of sets, $\Cat$ denotes the 2-category of categories, and $\Bicat$ denotes the tricategory of bicategories \cite{GPS95}. 

    \item
    To introduce elements in a (bi)category, we use the type-theory inspired notation ``$A \colon \mathcal{A}$'' rather than the usual set-theory inspired notation ``$A \in \mathcal{A}$''.
    We use a classical ``trickle down'' abbreviation of typing declarations, for example $f\colon A \to B \colon \cc{A}$ is short for $A,B \colon \cc{A}$ and $f: A \to B$.
    
    \item
    We write $f\colon A \simeq B\colon \mathcal{A}$ to denote an {equivalence} $f$ from $A$ to $B$ in the bicategory $\mathcal{A}$.
    We write \mbox{$\alpha\colon f \simeq g\colon A \to B\colon \cc{A}$} to denote an invertible 2-cell $\alpha$ from $f$ to $g$.
    Isomorphisms of (bi)categories are denoted by $\cong$.
    
    \item
    As in \cite{JS17}, we use $\yo$ (the first letter of “Yoneda” when written in hiragana) to denote Yoneda embeddings.
    
    \item
    For $\mathcal{A},\mathcal{B}\colon \Bicat$, we denote by $[\mathcal{A},\mathcal{B}]$ the bicategory of pseudo-functors, pseudo-natural transformations and modifications. This is the hom-bicategory $\Bicat(\cc{A},\cc{B})$ of the tricategory $\Bicat$.
\end{enumerate}

\subsection{A few bicategorical facts}\label{subsec:facts}
We study here bicategories as introduced by Bénabou \cite{Ben67}, following mostly the notation from the book ``2-Dimensional Categories'' written by Niles Johnson and Donald Yau \cite{JY20}.
We recall here the main definitions of bicategory theory to fix our notation for the rest of this paper.
We omit details which are ubiquitous in the literature, referring the reader to \cite{JY20}, among many other choices.

\begin{definition}[Bicategory]\label{def:bicategory}
A \textbf{bicategory} $\mathcal{B}\colon\Bicat$ is the data of objects, arrows between objects and 2-cells between parallel arrows together with:
\begin{enumerate}[label = $\bullet$]
    \item
    A \textbf{vertical composition} that we denote by $\circ$, which makes $\mathcal{B}(A,B)$ into a category for all $A,B\colon \mathcal{B}$. For each arrow $f\colon A \to B\colon\mathcal{B}$ we denote its {identity 2-cell} by $1_f\colon f \to f$.
    
    \item
    A \textbf{horizontal composition} which is functorial in the vertical composition:
    \[
        c_{A,B,C}\colon
            \mathcal{B}(B,C) \times \mathcal{B}(A,B)
            \longrightarrow
            \mathcal{B}(A,C) \colon
                \Cat
    \]
    We denote horizontal composition of arrows by $\circ$, and horizontal composition of 2-cells by $\star$. By a usual abuse of notation, we use the same symbol $\star$ for the {\em whiskering} of arrows with 2-cells: $\beta \star f \coloneqq \beta \star 1_f$ and $g \star \alpha \coloneqq 1_g \star \alpha$.
    
    We allow ourselves to omit the symbol $\circ$ both for the horizontal composition of arrows and the vertical composition of 2-cells.
    
    \item
    An \textbf{identity arrow} $1_A\colon A\to A$ for each object $A\colon\mathcal{B}$, \textbf{right} and \textbf{left unitors} $r_u\colon u1_A \simeq u$ and $l_u\colon 1_Bu \simeq u$ respectively for each $u\colon A \to B\colon \mathcal{B}$, and \textbf{associators} $a_{u,v,w}\colon (uv)w \simeq u(vw)$ for each triple $u,v,w$ of composable arrows.
    Unitors and associators are required to be natural, and the triangle and pentagon axioms are required to hold.
\end{enumerate}
\end{definition}

We will assume that the reader is familiar with the idea of pasting diagrams and with the coherence theorems of bicategory theory.

\begin{definition}[Functors between Bicategories]\label{def:pseudo-functor} A \textbf{pseudo-functor} $F\colon \mathcal{A} \to \mathcal{B}\colon \Bicat$ is the data of a mapping of objects to objects, a mapping of arrows to arrows and a mapping of 2-cells to 2-cells, compatible with domain and codomain, together with:
\begin{enumerate}[label = $\bullet$]
    \item
    A natural \textbf{functoriality constraint} $F^2_{u,v}\colon Fu \circ Fv \simeq F(u\circ v)$ for each pair $u,v$ of composable arrows.
    
    \item
    A \textbf{unity constraint} $F^0_A\colon 1_{FA} \simeq F(1_A)$ for each object $A\colon \mathcal{B}$.
\end{enumerate}
This data is subject to four axioms: local functoriality, the lax associativity, lax left and lax right unity axioms. 
For each $A,B\colon \mathcal{A}$ we denote the \emph{local hom-functors} by  $F_{A,B}\colon \mathcal{A}(A,B) \to \mathcal{B}(FA,FB)$.
\end{definition}

\begin{remark}[Orientation of Diagrams]
As much as possible, throughout this paper, arrows and 2-cells in diagrams are oriented downwards or rightwards.
In particular, when the symbol ``$\simeq$'' is used in a diagram to denote an invertible 2-cell, its orientation, if unclear, is assumed to be governed by this general principle.
\end{remark}

\begin{remark}[Equivalences]\label{rem:equivalences}\leavevmode
\begin{enumerate}[label = $\bullet$]
    \item
    We say that a pseudo-functor $F\colon\mathcal{A} \to \mathcal{B} \colon \Bicat$ is a \emph{biequivalence}, or equivalently an equivalence, when there exists a pseudo-functor $G:\mathcal{B} \to \mathcal{A}$ and natural equivalences $id_{\cc{A}} \simeq G F$, $F G \simeq id_{\cc{B}}$ in the bicategories $[\cc{A},\cc{A}]$ and $[\cc{B},\cc{B}]$.
    
    \item
    $F\colon \mathcal{A} \to \mathcal{B}$ is a biequivalence if and only if $F$ is \emph{essentially surjective} (for all $B\colon \mathcal{A}$, there exists $A\colon \mathcal{A}$ and $f\colon B\simeq FA$) and locally an equivalence (each $F_{A,B}$ is an equivalence).
    Note that this characterization of biequivalences requires the axiom of choice.
    
    \item
    Recall that a pseudo-natural transformation $\alpha \colon F \to B \colon \mathcal{A} \to \mathcal{B}\colon \Bicat$ is an equivalence in $[\cc{A},\cc{B}]$ if and only if $\alpha_A$ is an equivalence in $\mathcal{B}$ for each $A\colon \mathcal{A}$.
\end{enumerate}
\end{remark}

\subsection{On fibrations and 1-fibrations of bicategories}\label{subsec:fibrations}
We begin by recalling some basic definitions of the theory of fibrations of bicategories from \cite{Buc14}, in a way that fits our purpose.
We fix $P \colon \E \to \B$ a pseudo-functor between bicategories.
Recall that an arrow $f\colon X \to Y\colon \E$ is Cartesian if and only if, for every $Z\colon\mathcal{E}$, the canonical diagram
\begin{equation}\label{eq:cartesian-arrow}\xymatrix{
    \mathcal{E}(Z,X) \ar[r]^{f_*}
        \ar[d]_{P_{Z,X}}
        \ar@{}[dr]|{\Mreq{P^{2}_{f,-}}}
    & \mathcal{E}(Z,Y) \ar[d]^{P_{Z,Y}} \\
    \mathcal{B}(PZ,PX) \ar[r]_{(Pf)_*}
    & \mathcal{B}(PZ,PY)
}\end{equation}
exhibits the category $\mathcal{E}(Z,X)$ as a bipullback of $\mathcal{E}(Z,Y)$ and $\mathcal{B}(PZ,PX)$ over $\mathcal{B}(PZ,PY)$.
We denote by $\mathcal{C}_1$ the family of Cartesian arrows of $P$.

Recall also that $\Cat$ has pseudo-pullbacks, which are computed as follows: given a cospan of functors \mbox{${\bf A} \mr{F} {\bf B} \ml{G} {\bf C}$,} the objects in the pseudo-pullback ${\bf A} \times_{{\bf B}} {\bf C}$ are given by triples $(A,C,\alpha)$, where $FA \Mreq{\alpha} GC$ is an invertible arrow, and the arrows $(A,C,\alpha) \to (A',C',\alpha')$ by pairs $A \mr{f} A'$, $C \mr{f} C'$ such that $\alpha' (Ff) = (Gg) \alpha$.
Then, we can express the statement above that ``the canonical diagram exhibits the category $\mathcal{E}(Z,X)$ as a bipullback..." equivalently as ``the comparison morphism 
\begin{equation}\label{eq:from-Buckley-312}
    \mathcal{E}(Z,X)
    \xr{(f_*, P_{Z,X},P^{2}_{f,-})}
    \mathcal{E}(Z,Y)
    \times_{\mathcal{B}(PZ,PY)}
    \mathcal{B}(PZ,PX)
\end{equation}
is an equivalence of categories".
Writing explicitly what it means for the functor in \eqref{eq:from-Buckley-312} to be respectively essentially surjective, full, and faithful, we have the following detailed description of the properties of a Cartesian arrow, equivalent to the ones in \cite[Def. 3.1.1]{Buc14}.

\begin{lemma}\label{lem:cartesian-arrow}
An arrow $f\colon X \to Y\colon \E$ is Cartesian if and only if it satisfies
\begin{enumerate}[label=$\arabic*.$]\setcounter{enumi}{-1}
    \item
    For each $Z\colon \E$, $g \colon Z \to Y$, $h\colon PZ \to PX$ with an isomorphism $\alpha \colon Pf \circ h \simeq Pg$
    \begin{equation*}
        \vcenter{\xymatrix@C+15pt@R+15pt{
            Z  \ar[dr]^{g}
                \ar@/_1pc/@{.>}[d]_{\hat{h}}
            & \\
            X \ar[r]_{f}
                \ar@{}[ur]|(0.35){\Mreq{\hat{\alpha}}}
            & Y
        }}
        \quad\quad \textrm{\rotatebox{180}{$\leadsto$}} \quad\quad
        \vcenter{\xymatrix@C+15pt@R+15pt{
            PZ \ar@/^1pc/[d]|{h}
                \ar@/_1pc/@{.>}[d]_{P\hat{h}}
                \ar[dr]^{Pg}
                \ar@{}[d]|{\Mreq{\hat{\beta}}}
            & \\
            PX \ar[r]_{Pf}
                \ar@{}[ur]|(0.4){\Mreq{\alpha}}
            & PY
        }}
    \end{equation*}
    there exists an $\hat{h}\colon Z \to X$ and isomorphisms $\hat{\alpha} \colon f\hat{h} \simeq g$, $\hat{\beta} \colon P\hat{h} \simeq h$ such that $\alpha \circ (Pf\star\hat{\beta}) = P\hat{\alpha} \circ P^2_{h,f}$.
    We say that $(\hat{h},\hat{\alpha},\hat{\beta})$ is a \emph{lift} of $(h, \alpha)$, or that it \emph{factors} $g$ through (the Cartesian arrow) $f$ above $\alpha$.

    \item
    If $g, h\colon Z \to X$ are arrows in $\cc{E}$, and there are 2-cells $\alpha \colon fg \Rightarrow fh$ in $\cc{E}$, $\beta \colon Pg \Rightarrow Ph$ in $\cc{B}$, such that $P\alpha$ equals the composition $P(fg) \simeq Pf\circ Pg \Mr{Pf \star \beta} Pf\circ Ph \simeq P(fh)$, then there is a 2-cell $\widehat{\beta} \colon g \Rightarrow h$ in $\cc{E}$ such that  $P\widehat{\beta} = \beta$, $f \star \widehat{\beta} = \alpha$.
    
    \item
    If $\alpha, \beta\colon g \Rightarrow h\colon Z \to X$ are 2-cells in $\cc{E}$, then the equalities $P\alpha = P\beta$ and $f \star \alpha = f \star \beta$ together imply $\alpha = \beta$. In particular, this implies that the 2-cell $\hat{\beta}$ in item $1$ is unique.
    \qed
\end{enumerate} 
\end{lemma}

We begin by introducing a notion of 1-fibration, which is weaker than the notion of fibration introduced in \cite[Def. 3.1.5]{Buc14} and does not involve Cartesian 2-cells (see Definition \ref{def:fibration}).
This is done here mainly for a matter of convenience, because we noticed that Cartesian 2-cells are not needed for proving the Lifting Fractions Lemma \ref{lem:1-fibrationsliftfractions}.
In addition, we mention that some preliminary computations seem to indicate that 1-fibrations could be related to a kind of lax homomorphism into $\Bicat$, with an equivalence given by a Grothendieck construction, just like fibrations are related to trihomomorphisms in \cite{Buc14}.

\begin{definition}[1-Fibration]\label{def:1-fibration}
We say that $P$ is a \emph{1-fibration} if for any $E \colon \E$ and $f \colon B \to PE$, there exists an object $\hat{B}\colon \cc{B}$ and a Cartesian arrow $\hat{f} \colon \hat{B} \to E$ with $P\hat{B} = B$ and $P\hat{f} = f$:
\[
    \vcenter{\xymatrix{
        \hat{B} \ar[r]^{\hat{f}}
        & E
    }}
    \quad\quad \textrm{\rotatebox{180}{$\leadsto$}} \quad\quad
    \vcenter{\xymatrix{
        B \ar[r]^f
        & PE
    }}
\]
We refer to $\hat{f}$ as a \emph{Cartesian lift} of $f$ (at $E$).
\end{definition}
\noindent It is natural to ask in Definition \ref{def:1-fibration} for an invertible 2-cell $P\hat{f} \simeq f$ instead of the equality (see \cite[Rem. 3.1.6]{Buc14}), or even for an equivalence $P\hat{B} \simeq B$ and an invertible 2-cell filling the appropriate triangle.
A careful check of the proof of Lemma \ref{lem:1-fibrationsliftfractions} shows that the results still hold under those hypotheses.

We recall now from {\cite[3.1.4, 3.1.5]{Buc14}} what it means for the pseudo-functor $P \colon \E \to \B$ to be a fibration.

\begin{definition}\label{def:cartesian-2-cell}
We say that $\alpha\colon f \Rightarrow g \colon X \to Y \colon \cc{E}$ is Cartesian if it is Cartesian for the local functor $P_{X,Y}\colon \cc{E}(X,Y) \to \cc{B}(PX,PY)$. We denote by $\mathcal{C}_2$ the family of Cartesian 2-cells of $P$.
\end{definition}

\begin{definition}\label{def:fibration}
We say that $P$ is a \emph{fibration} when it is a locally fibred 1-fibration such that the horizontal composition of Cartesian 2-cells is Cartesian.
\end{definition}

The following are some basic properties of Cartesian arrows that hold when $P$ is a fibration (see {\cite[3.1.8 to 3.1.12]{Buc14}}) and that we will be using in this paper:

\begin{proposition}\label{prop:basicpropCartarrows}
Equivalences are Cartesian, and Cartesian arrows are closed both under composition and under invertible 2-cells.
\end{proposition}

\begin{proposition}\label{prop:onesided2for3Cartarrows}
Given composable arrows $f$ and $g$, if $g$ and $gf$ are Cartesian then so is $f$.
\end{proposition}

\begin{proposition}\label{prop:cartoverequivisequiv}
If $f$ is Cartesian and $Pf$ is an equivalence, then so is $f$.
\end{proposition}

\section{Filtered and Pseudofiltered Bicategories}\label{sec:filtered}
In this section we consider the filtered axioms for a 2-category that can be found in \cite{Ken92} and the pseudofiltered axioms 
introduced in \cite{DS06} and corrected in \cite{DS21}.
We give here the definitions of filtered (\ref{def:filtered}) and pseudofiltered (\ref{def:pseudofiltered}) bicategories in two equivalent ways: as a set of axioms and as a general statement about existence of pseudo-cocones.
By showing that the axioms introduced here are equivalent to the ones in the literature, it follows from our work that some of the requirements can be dropped from these axioms (see Lemma \ref{lem:1-flt-invertible-for-free} and Proposition \ref{prop:literature-pseudofiltered}).

We begin by recalling the definition of pseudo-cocone of a diagram in a bicategory (a version for 2-categories and 2-functors can be found, for example, in \cite[p.13]{DS21}).

\begin{definition}[Pseudo-Cocones]\label{def:pseudo-cocone}
By a \emph{diagram} (in $\cc{B}$) we mean a pseudo-functor $F\colon \mathcal{C} \to \mathcal{B}$, where $\mathcal{C}$ is a bicategory.
The diagram is said to be \emph{finite} when $\mathcal{C}$ is (that is, it has finitely many objects, arrows, and 2-cells). 
A \emph{pseudo-cocone} of a diagram, with apex $E \colon \mathcal{B}$ is given, as usual, by a pseudo-natural transformation $\theta\colon F \Rightarrow \Delta E$ into the constant pseudo-functor.
As such, it is given by two families of arrows and invertible 2-cells - $\{FA  \xr{\theta_A} E\}_{A\colon \cc{C}}$ and \mbox{$\{\theta_{B} Ff \Mreq{\theta_f} \theta_A \}_{A\xr{f} B \colon \cc{C}}$ -} satisfying the following three axioms that we recall here for the reader's convenience (note that \emph{a priori}, the codomain of $\theta_f$ should be $1_E\theta_A$ but we can write is as above by using implicitly the left unitor of $\mathcal{B}$):
\begin{enumerate}[align=left, leftmargin=*]
    \itemaxiom[PC0]
    For $A \colon \cc{C}$, we have a unity axiom
    \[
        \vcenter{\xymatrix@R=4em@C=4em{
            FA \ar@/_1pc/[d]_{1_{FA}}
                \ar@/_0.2pc/@{{}{ }{}}[d]|{\Mreq{F^0_A}}
                \ar@/^1pc/[d]|{F1_A}
                \ar@/^1pc/[dr]^{\theta_A}
                \ar@/_0.1pc/@{{}{ }{}}[dr]|(0.6){\Mreq{\theta_{1_A}}} \\
            FA \ar[r]_{\theta_A}
            & E
        }}
        \qquad = \qquad
        \vcenter{\xymatrix@R=4em@C=4em{
            FA \ar@/_1pc/[d]_{1_{FA}}
                \ar@/^1pc/[dr]^{\theta_A}
                \ar@/_1pc/@{{}{ }{}}[dr]|(0.4){\Mreq{r_A}} \\
            FA \ar[r]_{\theta_A}
            & E
        }}
    \]
    \itemaxiom[PC1]
    For $A \mr{f} B \mr{g} C \colon \cc{C}$, we have a functoriality axiom
    \[
        \vcenter{\xymatrix{
            FA \ar@/^1pc/[drrr]^{{\theta}_A} 
                \ar[d]_{Ff}^{\qquad\theta_f\wr\Uparrow}\\
            FB \ar[rrr]^{{\theta}_B}
                \ar[d]_{Fg}^{\qquad\theta_g\wr\Uparrow}
            &&& E\\
            FC \ar@/_1pc/[urrr]_{{\theta}_C}
        }}
        \qquad = \qquad
        \vcenter{\xymatrix{
            FA \ar@/^3pc/[dd]|{Fgf}
                \ar@/^1.5pc/@{{}{ }{}}[dd]|{\Mreq{F^2_{g,f}}}
                \ar@/^1pc/[drrr]^{{\theta}_A}
                \ar[d]_{Ff}\\
            FB \ar@{}[rrr]|(0.6){\theta_{gf}\wr\Uparrow}
                \ar[d]_{Fg}&&& E\\
            FC \ar@/_1pc/[urrr]_{{\theta}_C}
        }}
    \]
    
    \itemaxiom[PC2]
    For $\xymatrix{
        A \ar@{}@<.5pc>[r]|{}="1"
        \ar@/^1pc/[r]^{f}
        \ar@{}@<-.5pc>[r]^{}="2"
        \ar@/_0.7pc/[r]_{g}
        & B
        \ar@{=>}^<<{\gamma}"1";"2"
    }\colon \cc{C}$, we have a vertical naturality axiom
     \[
         \vcenter{\xymatrix{
            FA \ar@<1.4ex>[drr]^{{\theta}_A}
                \ar@/^-1.6ex/[d]^(0.4){F\gamma}^(0.6){\; \Rightarrow}_{Ff}  
                \ar@/^2.3ex/[d]^{Fg}^(0.6){\quad\;\;\theta_g\wr\Uparrow}\\
            FB \ar[rr]_{{\theta}_B}
            && E
        }}
        \qquad = \qquad
        \vcenter{\xymatrix{
            FA \ar[drr]^{{\theta}_A}
                \ar[d]_{Ff}^(0.6){\quad\theta_f\wr\Uparrow}\\
            FB \ar[rr]_{{\theta}_B}
            && E
        }}
    \]
\end{enumerate}
The dual notion is called \emph{pseudo-cone}.
\end{definition}

\subsection{On filtered bicategories}\label{subsec:filtered}
Now that we have explicitly stated the definition of pseudo-cocone, we can define and characterize filtered bicategories as follows:

\begin{definition}[Filtered Bicategory]\label{def:filtered}
A non-empty bicategory $\mathcal{B}$ is said to be \emph{filtered} if every finite diagram has a pseudo-cocone. 
$\mathcal{B}$ is said to be \emph{cofiltered} if $\mathcal{B}^{\textnormal{op}}$ is filtered; that is, if any finite diagram in $\mathcal{B}$ has a pseudo-cone.
\end{definition}

\begin{proposition}[Filtered Bicategory]\label{prop:flt-axioms}
A non-empty bicategory is filtered if and only if it satisfies the following three axioms:
\begin{enumerate}[align=left, leftmargin=*]
    \itemaxiom[0-Flt]
    For any objects $A,B$, there exist an object $C$ and arrows $u\colon A \to C$, $v\colon B \to C$,
    \[\xymatrix@R=0ex{
        A \ar@{.>}[dr]^u \\
        & C \\
        B \ar@{.>}[ur]_v
    }\]
    
    \itemaxiom[1-Flt]
    For any pair of parallel arrows $f,g\colon A \to B$, there exist an object $C$, an arrow $u\colon B \to C$, and a 2-cell $\gamma\colon uf \Rightarrow ug$,
    
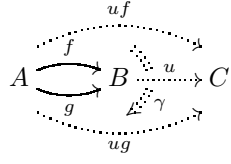
\[\xymatrix{
        A \ar@/^/[r]^f
            \ar@/_/[r]_g
            \ar@{.>}@<.7pc>@/^1pc/[rr]^{uf}="1"
            \ar@{}@<1.3pc>[rr]^{}="1"
            \ar@{.>}@<-.7pc>@/_1pc/[rr]_{ug}
            \ar@{}@<-1.3pc>[rr]^{}="2"
        & B \ar@{.>}[r]^u
        & C
        \ar@/^1pc/@{:>}^(0.7){\gamma}"1";"2"
    }\]
    
    \itemaxiom[2-Flt]
    For any pair of parallel 2-cells $\alpha,\beta\colon f \Rightarrow g\colon A \to B$, there exist an object $C$ and an arrow $u\colon B \to C$ such that $u \star \alpha = u \star \beta$,
    \[\xymatrix{
        A \ar@{}@<.5pc>[r]|{}="1"
            \ar@/^1pc/[r]^f
            \ar@{}@<-.5pc>[r]^{}="2"
            \ar@/_0.7pc/[r]_g
        & B \ar@{.>}[r]^u
        & C
        & =
        & A \ar@{}@<.5pc>[r]|{}="3"
            \ar@/^1pc/[r]^{f}
            \ar@{}@<-.5pc>[r]^{}="4"
            \ar@/_0.7pc/[r]_{g}
        & B \ar@{.>}[r]^u
        & C
        \ar@{=>}^<<{\alpha}"1";"2"
        \ar@{=>}^<<{\beta}"3";"4"
    }\]
\end{enumerate}

We denote the dual cofiltered axioms with the names $\ref{0-Flt}^{\textnormal{op}}$, $\ref{1-Flt}^{\textnormal{op}}$ and $\ref{2-Flt}^{\textnormal{op}}$.
\end{proposition}

We note that these three axioms appear in \cite{Ken92} for the case of 2-categories except that the 2-cell in Axiom \ref{1-Flt} was required to be invertible\footnote{One of our reasons not to keep the names BF0-BF2 from  \cite{Ken92} is to avoid confusion with the names of the axioms for a bicategory of fractions \cite{Pro96}.}. 
Before providing the proof of Proposition \ref{prop:flt-axioms}, we show that this requirement can be omitted in the presence of Axiom \ref{2-Flt}.
A similar situation was observed for an axiom for the right bicalculus of fractions (see \cite[Prop. 2.3]{PS21}).
We mention in passing to the interested reader that we have realized that with the same reasoning one can show that, in \cite[Def. 3.1.2]{DDS18}, the sentence ``If $f \in \Sigma$, we may choose $\alpha$ invertible" can be omitted as well.

\begin{lemma}\label{lem:1-flt-invertible-for-free}
Assuming \ref{2-Flt}, the 2-cell $\gamma$ in Axiom \ref{1-Flt} can be taken to be invertible.
\end{lemma}
\begin{proof}
Consider a pair of parallel arrows $f,g\colon A \to B$.
We use Axiom \ref{1-Flt} twice. First we apply this axiom to $f$ and $g$, which yields an arrow $u^1\colon B \to C^1$ and a 2-cell $\gamma^1\colon u^1f \Rightarrow u^1g$.
Then, we apply Axiom \ref{1-Flt} to $u^1g$ and $u^1f$, which yields an arrow $u^2\colon C^1 \to C'$ and, after composing with the associators, a 2-cell $\delta'\colon (u^2u^1)g \Rightarrow (u^2u^1)f$.
Taking $u' = u^2u^1\colon B \to C'$, and defining $\gamma'$ as a composition of $u^2 \star \gamma^1$ with the associators, we have a pair of 2-cells $\gamma'\colon u'f \Rightarrow u'g$ and $\delta'\colon u'g \Rightarrow u'f$.

We then use Axiom \ref{2-Flt}, also twice.
First we apply it to the 2-cells $\delta' \gamma'$, and $1_{u'f}$, which yields an arrow $v^1\colon C' \to C^2$ that co-equifies these 2-cells.
We then apply Axiom \ref{2-Flt} to the 2-cells $v^1 \star (\gamma'\delta')$ and $1_{v^1(u'g)}$, this yields an arrow $v^2\colon C^2 \to C$ that co-equifies these 2-cells.
Taking $v = v^2v^1\colon C' \to C$, it follows that $\gamma = v \star \gamma'$ and $\delta = v \star \delta'$ are the inverses of each other, and taking   $u = v u'$ concludes the proof.
\end{proof}

\begin{proof}[Proof of Proposition \ref{prop:flt-axioms}] 
The same ideas from \cite[Prop. 3.1.5]{DDS18} work, but for the sake of completeness we give this brief proof here.
For the $(\Rightarrow)$ direction, each of the three axioms in Proposition \ref{prop:flt-axioms} follows by considering pseudo-cones of the diagrams
\begin{center}
\addtolength{\tabcolsep}{-1pt} 
    \begin{tabular}{ccc}
        $\left\{\bullet \; \; \bullet\right\}
            \rightarrow
            \left\{A \; \; B \right\}
            \quad$
        & $\quad
            \left\{\xymatrix{
                \bullet \ar@/^/[r]
                    \ar@/_/[r]
                & \bullet
            }\right\}
            \rightarrow
            \big\{\xymatrix{
                A \ar@/^/[r]^f
                    \ar@/_/_g[r]
                & B
            }\big\}
            \quad$
        & $\quad
            \big\{\xymatrix@C=2.7em{
                \bullet \ar@/^0.6pc/[r]^{}="1"
                    \ar@/_0.6pc/[r]^{}="2"
                & \bullet
                \ar@{=>}@/^/"1";"2"
                \ar@{=>}@/_/"1";"2"
            }\big\}
            \rightarrow
            \big\{\xymatrix@C=2.7em{
                A \ar@{}@<0.6pc>[r]^{}="1" 
                    \ar@{}@<-0.6pc>[r]_{}="2" 
                    \ar@/^0.7pc/[r]^{f} 
                    \ar@/_0.7pc/[r]_{g}
                & B
                \ar@{=>}@/^/^{\beta}"1";"2"
                \ar@{=>}@/_/_{\alpha}"1";"2"
            }\big\}$ 
    \end{tabular} 
\addtolength{\tabcolsep}{1pt}
\end{center}
For the opposite direction $(\Leftarrow)$, given a finite pseudo-diagram $F\colon \mathcal{C} \to \mathcal{B}$, we proceed as follows:
\begin{enumerate}[label=$\arabic*.$]\setcounter{enumi}{-1}
    \item
    First, we build a pseudo-cocone on the objects of the diagram; that is, a collection of arrows $FA \mr{\theta_A} E$, one for each $A\colon \cc{C}$, by using \ref{0-Flt} successively. 
    
    \item
    We consider then each arrow of $\cc{C}$.
    Given such an arrow $A \mr{f} B$, we use axiom \ref{1-Flt} with the two parallel arrows $\theta_A$ and $\theta_B Ff$, this gives an arrow $E \mr{u} E'$ and a 2-cell $\theta_f$ (that can be taken invertible by Lemma \ref{lem:1-flt-invertible-for-free}).
    We then ``update'' our cocone by composing it with $u$, in other words we rename $E \coloneqq E'$, $\theta_A \coloneqq u \theta_A$ for all $A$, and $\theta_g \coloneqq u \theta_g$ for any previously constructed 2-cell $\theta_g$.
    After considering all the arrows of the diagram one by one, we obtain a pseudo-cocone on the objects and arrows of the diagram.
    
    \item
    Finally we consider each of the equations coming from the axioms \refaxioms{PC0}{PC0-2} in Definition \ref{def:pseudo-cocone}.
    We proceed in a similar manner to step 1, but using \ref{2-Flt} instead of \ref{1-Flt}, to make our pseudo-cocone satisfy these finitely many equations.
\end{enumerate}
\vspace{-.5cm}
\end{proof}

\begin{remark}[Why not weighted cocones?]\label{rem:why-not-weighted}
The notion of conical cocone in bicategory theory is not sufficient for all applications.
Hence it is a natural question to wonder if one could define other forms of \textit{filteredness} using other, more exhaustive, types of cocones.
Even though this level of generalization might be useful to consider, in this specific case, the conical cocones are enough.
Indeed any conical cocone is also a weighted cocone on the same diagram for any weight, as the conical weight is terminal.
Hence, if we were to define a filtered bicategory with the family of all \textit{finite} weighted diagrams having a cocone, we would get an equivalent definition.
\end{remark}

\subsection{On pseudofiltered bicategories}\label{subsec:pseudofiltered}
A bicategory is said to be connected if it is connected as a graph; that is, for any pair of objects in it, there is a ``\emph{zig-zag}'' (a finite sequence of forward and backward arrows) linking them.
A diagram $\mathcal{C} \to \mathcal{B}$ is said to be connected if $\cc{C}$ is.

\begin{warning}
The use of the prefix ``pseudo" (without hyphen) in the following definition of pseudofiltered bicategory comes from the 1-dimensional case \cite[(SGA 4, Tome 2)]{GSV72}. It weakens the notion of filtered by requiring the diagram to be connected.
This should not be confused with the use of the same prefix (with hyphen) in the word ``pseudo-cocone" in the same definition.
Recall that the notion of pseudo-cocone is given in Definition \ref{def:pseudo-cocone}: this is the usual use of the prefix ``pseudo-" in bicategory theory, as in pseudo-functor and pseudo-natural transformation.
\end{warning}

\begin{definition}[Pseudofiltered Bicategory]\label{def:pseudofiltered}
A non-empty bicategory $\mathcal{B}$ is said to be \emph{pseudofiltered} if every finite \emph{connected} diagram has a pseudo-cocone.
$\mathcal{B}$ is said to be \emph{pseudocofiltered} if $\mathcal{B}^{\textnormal{op}}$ is pseudofiltered.
\end{definition}

Almost by definition, we have

\begin{proposition}[Filtered = Pseudofiltered + Connected]\label{prop:connected}
A non-empty bicategory is filtered if and only if it is pseudofiltered and connected.
\qed
\end{proposition}

\begin{proposition}[Pseudofiltered Bicategory]\label{prop:pflt-axioms}
A bicategory is pseudofiltered if and only if it satisfies the following three axioms:
\begin{enumerate}[label= $\bullet$,align=left, leftmargin=*]
    \itemaxiom[0-pFlt]
    For any objects $A,B,C$ and arrows $u\colon C \to A$, $v\colon C \to B$, there exist an object $D$ and arrows $r\colon A \to D$, $s\colon B \to D$:
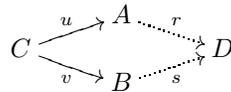
    \[\xymatrix@R=0ex{
        & A \ar@{.>}[dr]^r \\
        C \ar[ru]^{u}
            \ar[rd]_v
        && D \\
        & B \ar@{.>}[ur]_s
    }\]
    
    \itemaxiom[1-pFlt]
    Axiom \ref{1-Flt} in Proposition \ref{prop:flt-axioms}.
    
    \itemaxiom[2-pFlt]
    Axiom \ref{2-Flt} in Proposition \ref{prop:flt-axioms}.
    \qed
\end{enumerate}
\end{proposition}

As we only changed the names of \ref{1-Flt} and \ref{2-Flt}, Lemma \ref{lem:1-flt-invertible-for-free} still applies here and we have:
\begin{corollary}\label{coro:1-pflt-invertible-for-free}
Assuming \ref{2-pFlt}, the 2-cell $\gamma$ in Axiom \ref{1-pFlt} can be taken to be invertible.
\qed
\end{corollary}

And we also have, with an evident application of \ref{1-pFlt}:
\begin{lemma}\label{lem:0-pflt-commutes-for-free}
Under the asumption of \ref{1-pFlt} and \ref{2-pFlt}, one can add the existence of an (invertible) 2-cell $\gamma\colon ru \Rightarrow sv$ to the consequence of \ref{0-pFlt}.
\qed
\end{lemma}

\begin{remark}[A few other definitions and axioms]\label{rem:literature-filteredness}\leavevmode
\begin{enumerate}[label = $\bullet$]
    \item
    In \cite{DS06}, the authors give a definition of a pre-2-filtered 2-category using two axioms \textbf{F1} and \textbf{F2}.
    This is a notion that, while weaker than that of pseudo-2-filtered, still admits a construction of pre-2-filtered pseudo-colimits of categories that is analogous to the construction of (pseudo)filtered colimits of sets.
    The notion of pseudo-2-filtered 2-category is however needed in order to show that these pseudo-colimits commute with finite connected cotensors \cite[Th.~2.4]{DS06}.
    
    \item
    In \cite{DS06} the authors define pseudo-2-filtered 2-categories by strengthening \textbf{F1} to \textbf{FF1}, while keeping \textbf{F2}.
    In  \cite{DS21}, they correct this definition so that \cite[Th.~2.4]{DS06} holds, by adding a third axiom, \textbf{F3}.
    
    \item
    All these axioms can be considered for a bicategory instead of a strict 2-category.
    We outline below how one can show that the notion of pseudofiltered bicategory in Definition \ref{def:pseudofiltered} is equivalent to the one in \cite{DS21}, obtaining along the way that axiom \textbf{F2} can in fact be dropped from this definition.
\end{enumerate}
\end{remark}

\begin{exercise}\leavevmode
\begin{enumerate}
    \item
    Assume that $\mathcal{B}$ is pseudofiltered.
    Show that each of the axioms \textnormal{\textbf{FF1}}, \textnormal{\textbf{F2}}, and \textnormal{\textbf{F3}} given in \cite{DS21} holds in $\mathcal{B}$, by using a pseudo-cocone of a finite connected diagram.
    
    \item
    Show that axiom \textnormal{\textbf{FF1}} immediately implies both \ref{0-pFlt} and \ref{1-pFlt}.
    Note, as mentioned in \cite[Remark in p.240]{DS21}\,\footnote{In private correspondence, E. Dubuc has made the following clarifications to the content of \cite{DS21}: the 2-cell $\eps$ in axiom \textnormal{\textbf{F3}} is required to be invertible, and axiom \textnormal{\textbf{F2}} is not needed for proving the Remark on p.240.} that axiom \textnormal{\textbf{F3}} immediately implies \ref{2-pFlt}.
\end{enumerate}
\end{exercise}

After doing the exercise above, the reader will have shown:
\begin{proposition}
\label{prop:literature-pseudofiltered}
The following are equivalent sets of axioms, each expressing the property that $\mathcal{B}$ is pseudofiltered:
\begin{enumerate}[label = $(\roman*)$]
    \item
    \ref{0-pFlt}, \ref{1-pFlt}, and \ref{2-pFlt} in Proposition  \ref{prop:pflt-axioms}.
    
    \item
    \textnormal{\textbf{FF1}}, \textnormal{\textbf{F2}}, and \textnormal{\textbf{F3}} in \cite{DS21}.
    
    \item
    \textnormal{\textbf{FF1}} and \textnormal{\textbf{F3}} in \cite{DS21}.
    \qed
\end{enumerate}
\end{proposition}

\section{Diagrammatic Axioms of Fractions}\label{sec:diagrammatic}
In this section, we introduce a new set of axioms for a right bicalculus of fractions for a family $\cc{W}$ of arrows of a bicategory $\cc{B}$, a notion originally introduced in \cite{Pro96}.
This new set of axioms is then used to prove a series of lemmas, that are important for the rest of this paper, with simpler computations than the ones that would be required with the original axioms.
We show in Proposition \ref{prop:strongequivaxiom} that both sets of axioms are equivalent, so that when these axioms are satisfied the localization of $\cc{B}$ at $\cc{W}$ can be constructed as in \cite{Pro96}.
We begin by recalling the results as can be found in \cite{Pro96}, except that we give the following more general definition of a localization simultaneously at a family of 1-cells and at a family of 2-cells, since this notion will be used in the present paper (see Theorem \ref{thm:tricolimit}).

\begin{definition}[Localization by a family of arrows and 2-cells] \label{def:localizationformodel}
Let $\mathcal{B}$ be a bicategory, $\mathcal{W}_1$ be a family of arrows of $\mathcal{B}$ and $\mathcal{W}_2$ be a family of 2-cells of $\mathcal{B}$.
We say that a pseudo-functor $L\colon \mathcal{B} \to \mathcal{C}$ is a localization of $\mathcal{B}$ with respect to $\mathcal{W}_{1,2}$ if:
\begin{enumerate}[label = $\bullet$]
    \item
    it maps the arrows of $\mathcal{W}_1$ to equivalences and the 2-cells of $\mathcal{W}_2$ to invertible 2-cells, and furthermore,
    
    \item
    for each bicategory $\mathcal{D}$, the precomposition with $L$,
    \[\xymatrix{
        [\cc{C},\cc{D}] \ar[r]^-{L^*}
        & [\cc{B},\cc{D}]_{\cc{W}_{1,2}}
    }\]
    is a biequivalence of bicategories.
\end{enumerate}
Here $[\cc{B},\cc{D}]_{\cc{W}_{1,2}}$ stands for the full sub-bicategory of $[\cc{B},\cc{D}]$ on those pseudo-functors that send the arrows of $\cc{W}_1$ to equivalences and 2-cells of $\mathcal{W}_2$ to invertible 2-cells.
If such a localization exists, it is unique up to biequivalence and pseudo-natural isomorphism.
We usually denote $\mathcal{C}$ by $\mathcal{B}[\cc{W}_{1,2}^{-1}]$ and $L$ by $L_{\cc{W}_{1,2}}$.
Similarly, we define a localization of a bicategory $\mathcal{B}$ by a family of arrows $\mathcal{W}$, using the empty family of 2-cells in the above (or any subfamily of the family of invertible 2-cells).
For simplicity, we then write $[\cc{B},\cc{D}]_{\cc{W}}$, $\mathcal{B}[\cc{W}^{-1}]$ and $L_{\cc{W}}$ omitting the empty family of 2-cells.
\end{definition}

The first time a localization for bicategories was considered is in \cite{Pro96}, where a set of conditions \axiom[BF]{BF1-BF5} is given that allows for a construction of the localization of a family of arrows by a ``right bicalculus of fractions'', generalizing the one in \cite{GZ67}.
We record this result below:

\begin{definition}[Calculus of Fractions]\label{def:fractions}
We say that the pair $(\mathcal{B},\cc{W})$ (or the family $\cc{W}$) admits a \emph{bicalculus of right fractions}, or satisfies \emph{right fractions}, when the conditions \refaxioms{BF}{BF1-BF5} in \cite{Pro96} are satisfied.
We say that $(\mathcal{B},\cc{W})$ satisfies \emph{left fractions} if $(\cc{B}^{\textnormal{op}},\cc{W}^{\textnormal{op}})$ satisfies right fractions, and we say that $\cc{W}$  satisfies \emph{fractions} if it satisfies right or left fractions.
\end{definition}

\begin{theorem}\label{thm:fractions}
Let $\mathcal{B}$ be a bicategory and $\mathcal{W}$ be a family of arrows admitting a bicalculus of fractions.
Then the localization of $\mathcal{B}$ by $\mathcal{W}$ exists.
\qed
\end{theorem}

Note that in the original \cite{Pro96}, the universal property of the bicategory of fractions is stated slightly differently from the one we give in Definition \ref{def:localizationformodel} above, but in \cite{PS21} both properties are seen to coincide, see \cite[Remarks 3.7 (1), (2)]{PS21}.

\subsection{A new set of diagrammatic axioms of fractions}\label{subsec:FrcAxioms}
We fix throughout this section a class $\cc{W}$ of arrows of a bicategory $\mathcal{B}$ that contains the equivalences (\refaxioms{BF}{BF1}), is stable under composition (\refaxioms{BF}{BF2}), and is closed under invertible 2-cells (\refaxioms{BF}{BF5}).
We use the symbol $\wto$ to denote arrows in $\cc{W}$.
We present now a set of three axioms that can \emph{replace} the remaining axioms {\bf 3} and {\bf 4} from \cite{Pro96} (or from \cite{PS21}), that is the {\em diagrammatic} axioms.

\begin{enumerate}[align=left, leftmargin=*]
    \itemaxiom[0-Frc]
    Given objects $A,B,C$ and arrows $w\colon A \wto B \in \mathcal{W}$, $f\colon C \to B$, there exist an object $D$, arrows $u\colon D \wto C \in \mathcal{W}$, $h\colon D \to A$, and an invertible 2-cell $\alpha\colon fu \simeq wh$
    \[\xymatrix{
        D \ar@{.>}[r]^h
            \ar@{.>}[d]_u|{\circ}
            \ar@{}[dr]|{\Mreq{\alpha}}
        & A\ar[d]^w|{\circ}\\
        C \ar[r]_f
        & B
    }\]
    
    \itemaxiom[1-Frc]
    Given objects $B,C,D$, arrows $w\colon C \wto D \in \mathcal{W}$, $f,g\colon B \to C$, and a 2-cell $\alpha\colon wf \Rightarrow wg$, there exist an object $A$, an arrow $u\colon A \wto B \in \mathcal{W}$, and a 2-cell $\beta\colon fu \Rightarrow gu$ such that $a_{w,g,u} \circ (\alpha \star u) = (w \star \beta) \circ a_{w,f,u}$:
    \[
        \xymatrix{
            B \ar@/^/[r]^f
                \ar@/_/[r]_g
                \ar@<.7pc>@/^1pc/[rr]^{wf}
                \ar@{}@<1.3pc>[rr]^{}="1"
                \ar@<-.7pc>@/_1pc/[rr]_{wg}
                \ar@{}@<-1.3pc>[rr]^{}="2"
            & C \ar[r]^w|{\circ}
            & D
            \ar@/^1pc/@{=>}^(0.7){\alpha}"1";"2"
        }
        \qquad \leadsto \qquad
        \xymatrix{
            A\ar@{.>}[r]^u|{\circ} 
                \ar@{.>}@<.7pc>@/^1pc/[rr]^{fu}
                \ar@{}@<1.3pc>[rr]^{}="1"
                \ar@{.>}@<-.7pc>@/_1pc/[rr]_{gu}
                \ar@{}@<-1.3pc>[rr]^{}="2"
            & B \ar@/^/[r]^f
                \ar@/_/[r]_g
            & C
            \ar@/_1pc/@{:>}_(0.7){\beta}"1";"2"
        }
    \]
    \[\xymatrix{
        A \ar@{.>}[r]^u|{\circ}
        & B \ar@/^1pc/[r]^{wf}
            \ar@{}@<.5pc>[r]|{}="1"
            \ar@/_0.7pc/[r]_{wg}
            \ar@{}@<-.5pc>[r]|{}="2"
        & D
        & =
        & A \ar@{.>}@/^1pc/[r]^{fu}
            \ar@{}@<.5pc>[r]|{}="3"
            \ar@{.>}@/_0.7pc/[r]_{gu}
            \ar@{}@<-.5pc>[r]|{}="4"
        & C \ar[r]^w|{\circ}
        & D
        \ar@{=>}^{\alpha}"1";"2"
        \ar@{:>}^{\beta}"3";"4"
    }\]

    \itemaxiom[2-Frc]
    For any objects $B,C,D$, arrows $w\colon C \wto D \in \mathcal{W}$, $f,g\colon B \to C$, and 2-cells $\alpha, \beta\colon f \Rightarrow g$,  such that $w \star \alpha = w \star \beta$, there exist an object $A$ and an arrow $u\colon A \wto B \in \mathcal{W}$, such that $\alpha \star u = \beta \star u$:
    \[
        \xymatrix{
            B \ar@/^1pc/[r]^f_{}="1"
                \ar@/_1pc/[r]_g^{}="2"
            & C \ar[r]^w|{\circ}
            & D
            \ar@/^/@{=>}^{\alpha}"1";"2"
            \ar@/_/@{=>}_{\beta}"1";"2"
        }
        \qquad \leadsto \qquad
        \xymatrix{
            A \ar@{.>}[r]^u|{\circ}
            & B \ar@/^1pc/[r]^f_{}="1"
                \ar@/_1pc/[r]_g^{}="2"
            & C
            \ar@/^/@{=>}^{\alpha}"1";"2"
            \ar@/_/@{=>}_{\beta}"1";"2"
        }
    \]
\end{enumerate}

\begin{remark} \label{rem:beforeproof}
Note that the axiom \ref{0-Frc} is the same as axiom \refaxioms{BF}{BF3} in \cite{Pro96}.
Also, the content of axiom \ref{2-Frc} is shown to follow from the \ref{BF} set of axioms in \cite[Lemma 2.1]{Tom16} (see also \cite[Lemma 2.5]{PS21}).
Finally, note that axiom \refaxioms{BF}{BF4} has a first part which is exactly the content of axiom \ref{1-Frc}, a statement about the invertibility of $\beta$ that can be omitted (as shown in \cite[Lemma 2.3]{PS21}), and a final part that we recall now for convenience:
\begin{enumerate}
    \item[]
    {\em Furthermore, the collection of triples $(A, u, \beta)$ such that $\alpha \star u = w \star \beta$ as in \ref{1-Frc} satisfies the following property: for any two such triples $(A_1, u_1, \beta_1)$, $(A_2, u_2, \beta_2)$, there exists an object $X\colon\mathcal{B}$, arrows $s\colon X \to A_1$ and $t\colon X \to A_2$, and an invertible 2-cell $\eps\colon u_1 s \simeq u_2 t$ such that $u_1 s$ and $u_2 t$ are in $\cc{W}$ and such that we have the following equality of pastings:}
    \begin{equation}\label{eq:fromBF4}
        \vcenter{\xymatrix{
            \cdot \ar[r]^{t}
                \ar[d]_{s}
                \ar@{}[dr]|{\Mreq{\eps}}
            & \cdot \ar[d]^{u_2}|{\circ}
                \ar[r]|{\circ}^{u_2}
                \ar@{}[dr]|{\Mr{\beta_2}}
            & \cdot \ar[d]^{g}\\
            \cdot \ar[r]_{u_1}|{\circ}
            & \cdot \ar[r]_f
            & \cdot
        }}
        \qquad = \qquad
        \vcenter{\xymatrix{
            \cdot \ar[r]^{t}
                \ar[d]_{s}
                \ar@{}[dr]|{\Mreq{\eps}}
            & \cdot \ar[d]^{u_2}|{\circ}\\
            \cdot \ar[r]_{u_1}|{\circ}
                \ar[d]|{\circ}_{u_1}
                \ar@{}[dr]|{\Mr{\beta_1}}
            & \cdot \ar[d]^g\\
            \cdot \ar[r]_f
            & \cdot
        }}
    \end{equation}
\end{enumerate}
\end{remark}

\begin{proposition}[Equivalence of Diagrammatic Axioms]\label{prop:strongequivaxiom}
Assuming the closure Axioms \refaxioms{BF}{BF1}, \refaxioms{BF}{BF2} and \refaxioms{BF}{BF5}, we have the equivalence
\begin{equation*}
    \left(\refaxioms{BF}{BF3}+\refaxioms{BF}{BF4}\right)
    \iff
    \left(\ref{0-Frc}+\ref{1-Frc}+\ref{2-Frc}\right)
\end{equation*}
\end{proposition}
\begin{proof}
In view of Remark \ref{rem:beforeproof}, it only remains to show that the final part of the axiom \refaxioms{BF}{BF4} follows from the  \refaxioms{0-Frc}{Frc} set of axioms.
Pick two triples $(a_1, u_1, \beta_1)$, $(a_2, u_2, \beta_2)$ as in \ref{1-Frc}.
Applying first \ref{0-Frc} to $u_1$ and $u_2$, we get 
\[\xymatrix{
    \cdot \ar@{.>}[r]^{t'}
        \ar@{.>}[d]_{s'}
        \ar@{}[dr]|{\Mreq{\eps'}}
    & \cdot \ar[d]^{u_2}|{\circ} \\
    \cdot \ar[r]_{u_1}|{\circ}
    & \cdot
}\]
such that both $u_1 s'$ and $u_2 t'$ are in $\cc{W}$. As we have the following equalities of pasting:
\[
    \vcenter{\xymatrix{
        \cdot \ar[r]^{t'}
            \ar[d]_{s'}
            \ar@{}[dr]|{\Mreq{\eps'}}
        & \cdot \ar[d]^{u_2}|{\circ}
            \ar[r]|{\circ}^{u_2}
            \ar@{}[dr]|{\Mr{\beta_2}}
        & \cdot \ar[d]^{g} \\
        \cdot \ar[r]_{u_1}|{\circ}
        & \cdot \ar[r]_f
        & \cdot \ar[d]|{\circ}^{w} \\
        && \cdot
    }}
    \qquad = \qquad
    \vcenter{\xymatrix{
        \cdot \ar[r]^{t'}
            \ar[d]_{s'}
            \ar@{}[dr]|{\Mreq{\eps'}}
        & \cdot \ar[d]^{u_2}|{\circ} \\
        \cdot \ar[r]_{u_1}|{\circ}
        & \cdot \ar@/_1pc/[dr]_{wf}
            \ar@/^1pc/[dr]^{wg}
            \ar@{}[dr]|{\Mr{\alpha}} \\
        && \cdot
    }}
    \qquad = \qquad
    \vcenter{\xymatrix{
        \cdot \ar[r]^{t'}
            \ar[d]_{s'}
            \ar@{}[dr]|{\Mreq{\eps'}}
        & \cdot \ar[d]^{u_2}|{\circ} \\
        \cdot \ar[r]_{u_1}|{\circ}
            \ar[d]|{\circ}_{u_1}
            \ar@{}[dr]|{\Mr{\beta_1}}
        & \cdot \ar[d]^g \\
        \cdot \ar[r]_f
        & \cdot \ar[r]|{\circ}_{w}
        & \cdot
    }}
\]
we get, by \ref{2-Frc}, that the two 2-cells in \eqref{eq:fromBF4} can be made equal by precomposing with an arrow $u \in \cc{W}$, that is 
\[
    \vcenter{\xymatrix{
        \cdot \ar[d]|{\circ}_{u} \\
        \cdot \ar[r]^{t'}
            \ar[d]_{s'}
            \ar@{}[dr]|{\Mreq{\eps'}}
        & \cdot \ar[d]^{u_2}|{\circ}
            \ar[r]|{\circ}^{u_2}
            \ar@{}[dr]|{\Mr{\beta_2}}
        & \cdot \ar[d]^{g} \\
        \cdot \ar[r]_{u_1}|{\circ}
        & \cdot \ar[r]_f
        & \cdot
    }}
    \qquad = \qquad
    \vcenter{\xymatrix{
        \cdot \ar[r]|{\circ}^u
        & \cdot \ar[r]^{t'}
            \ar[d]_{s'}
            \ar@{}[dr]|{\Mreq{\eps'}}
        & \cdot \ar[d]^{u_2}|{\circ} \\
        & \cdot \ar[r]_{u_1}|{\circ}
            \ar[d]|{\circ}_{u_1}
            \ar@{}[dr]|{\Mr{\beta_1}}
        & \cdot \ar[d]^g \\
        & \cdot \ar[r]_f
        & \cdot
    }}
\]
We then take $s = s'u$, $t = t'u$ and $\eps = a_{u_2,t',u} \circ (\eps' \star u) \circ a^{-1}_{u_1,s',u}$ and the equation we get is precisely the one in \eqref{eq:fromBF4}.
\end{proof}

\subsection{On one- and two-dimensional calculi of fractions}\label{subsec:commutativitypi_0}
We introduce here the following lemma, that will allow us to compute pseudofiltered pseudo-colimits of categories using only the ordinary calculus of fractions from \cite{GZ67}.
We will denote these well-known axioms for right fractions by \axiom{R0} (wide subcategory), \axiom{R1} (Ore condition) and \axiom{R2} (existence of equalizing arrows), we refer to \cite{GZ67} or \cite[\S 3]{Fri11} for details.
We consider $\pi_0$, the left adjoint to the inclusion $d\colon\Cat\hookrightarrow\Bicat$.
We note that \ref{2-Frc} plays no role in the following proof.

\begin{lemma}\label{lem:pi0}
Let $\mathcal{B}$ be a bicategory and $\mathcal{W}$ a family of arrows of $\mathcal{B}$  satisfying right fractions.
Let $\mathcal{W}_0 = [\cc{W}]$ be the family of arrows of $\pi_0\mathcal{B}$ given by the equivalence classes containing arrows of $\mathcal{W}$.
Then $\mathcal{W}_0$ admits a (one-dimensional) calculus of right fractions in the category $\pi_0\mathcal{B}$.
\end{lemma}
\begin{proof}
We will write $[f]\colon x \to y$ for the equivalence class of an arrow $f\colon x \to y$ in $\pi_0\mathcal{B}$.
An arrow of $\mathcal{W}_0$ is, by definition, of the form $[w]$ for some $w \in \mathcal{W}$.
However, the reader will notice that even if $f$ is not in $\mathcal{W}$, it's possible that $[f]$ is in $\mathcal{W}_0$.
$\mathcal{W}_0$ satisfies \ref{R0} as $\cc{W}$ contains identities and is stable under composition.
Also, \ref{0-Frc} implies that $\mathcal{W}_0$ satisfies \ref{R1}. Hence the only technical step is proving \ref{R2}.
Let $f,g\colon b \to c$ and $w\colon c \wto d \in \mathcal{W}$ be arrows of $\mathcal{B}$ such that \[[wf] = [wg]\]
We want to show that there is an arrow $u\colon a \wto b \in \mathcal{W}$ such that $[fu] = [gu]$.
We know that the relation $[wf] = [wg]$ corresponds to the existence of a {\em zig-zag} of 2-cells from $wf$ to $wg$.
More precisely, there is an integer $n\geq 1$, a sequence of parallel arrows $h_0,\dots, h_n\colon b \to d$ with $h_0 = wf$, $h_n = wg$ and a sequence of 2-cells each oriented either as $\alpha_{i}\colon h_{i} \Rightarrow h_{i+1}$ or as $\alpha_{i}\colon h_{i+1} \Rightarrow h_{i}$, for $i\in \{0,\dots,n-1\}$.
If all arrows $h_i$ were of the form $wk_i$, then we could easily conclude by applying \ref{1-Frc} successively, but this is not necessarily the case.
We solve this problem by using \ref{0-Frc} in the following proof by induction on $n$:

If $n=1$, as stated above, we're done by \ref{1-Frc}.
If $n>1$, we apply the axiom \ref{0-Frc} to the arrows $w \in \mathcal{W}$ and $h_1$ to obtain a square,
\begin{equation}\tag{\ref{0-Frc}}\vcenter{\xymatrix{
    p \ar@{.>}[r]^{h^{\prime}}
        \ar@{.>}[d]_v|{\circ}
        \ar@{}[dr]|{\Mreq{\alpha}}
    & c \ar[d]^w|{\circ} \\
    b \ar[r]_{h_1}
    & d
}}\end{equation}
This will allow us, in the rest of this proof, to replace $h_1$ with $h_1v$, which, using the 2-cell above, can itself be seen as being of the form $wh'$ as desired. We do this by whiskering the 2-cells $\alpha_i$ with $v$, for $i\in \{1,\dots,n-1\}$, and we construct in this way a zig-zag of length $n-1$ between $(wg)v$ and $h_1v$.
By adding the invertible 2-cell $h_1v \simeq wh^{\prime}$ on one side and the associator on the other side of this zig-zag, we then get a zig-zag of length $n-1$ between $w(gv)$ and $wh^{\prime}$.
By the induction hypothesis, we have an object $a_1\colon\mathcal{B}$ and an arrow $u_1\colon a_1 \wto p \in \mathcal{W}$ such that $[(gv)u_1] = [h^{\prime}u_1]$:
\begin{equation}\label{eq:A}
    \vcenter{\xymatrix{
        a_1 \ar@{.>}[r]^{u_1}|{\circ}
        & p \ar@/^/[r]|{gv} \ar@/_/[r]|{h^{\prime}}
            \ar@/^1.8pc/[rr]^{w(gv)}|(0.6){}="0"
            \ar@{.}@/^1.3pc/[rr]|(0.58){}="1"
            \ar@{.}@/^0.8pc/[rr]|(0.56){}="2"
            \ar@{.}@/_0.8pc/[rr]|(0.56){}="3"
            \ar@{.}@/_1.3pc/[rr]|(0.58){}="4"
            \ar@/_1.8pc/[rr]_{wh^{\prime}}|(0.6){}="5"
        & c \ar[r]^{w}|{\circ}
        & d
        \ar@{=>}"0";"1"
        \ar@{=>}"2";"1"
        \ar@/_0.1pc/@{.}"2";"3"
        \ar@{=>}"3";"4"
        \ar@{=>}"5";"4"
    }}
    \qquad
    \stackrel{ind. hyp.}{\leadsto}
    \qquad
    [(gv)u_1] = [h^{\prime}u_1]
\end{equation}

We are now only left to deal with the 2-cell $\alpha_0$ and, to do so, we construct first a 2-cell $\gamma$ between $wh'$ and $w(fv)$ using $\alpha_0\star v$ and $\alpha$ as follows.
Note that $\alpha_0$ is oriented in an arbitrary direction, so $\gamma$ can either be $a_{w,f,v} \circ (\alpha_0\star v) \circ \alpha$ or $\alpha^{-1} \circ (\alpha_0\star v) \circ a_{w,f,v}^{-1}$ depending on the direction of $\alpha_0$.
Now by \ref{1-Frc} on $\gamma$ (independently of its direction, which below is drawn downwards but could be upwards), we have an object $a_0\colon\mathcal{B}$ and $u_0\colon a_0 \wto p \in \mathcal{W}$ such that $[(fv)u_0] = [h^{\prime}u_0]$:
\begin{equation}\label{eq:B}
    \vcenter{\xymatrix{
        a_0 \ar@{.>}[r]^{u_0}|{\circ}
        & p \ar@/^/[r]|{h^{\prime}}
            \ar@/_/[r]|{fv}
            \ar@/^1.8pc/[rr]^{wh^{\prime}}
            \ar@{}@<1.5pc>[rr]^{}="1" 
            \ar@/_1.8pc/[rr]_{w(fv)}
            \ar@{}@<-1.5pc>[rr]^{}="2"
        & c \ar[r]^{w}|{\circ}
        & d
        \ar@/^1pc/@{=>}^(0.8){\gamma}"1";"2"
    }}
    \qquad \stackrel{\ref{1-Frc}}{\leadsto} \qquad
    [(fv)u_0] = [h^{\prime}u_0]
\end{equation}
Finally, applying \ref{0-Frc} to $u_0$ and $u_1$, we get an object $a$ and an arrow $u\colon a \wto p \in \mathcal{W}$ that factors, up to invertible 2-cells, through both $u_0$ and $u_1$ and hence such that
\[
    [g(vu)]
    =
    [(gv)u]
    \stackrel{\eqref{eq:A}}{=}
    [h^{\prime}u]
    \stackrel{\eqref{eq:B}}{=}
    [(fv)u]
    =
    [f(vu)]
\]
Noting that $vu \in \mathcal{W}$, this finishes the proof.
\end{proof}

\subsection{Lifting fractions and filtered axioms through fibrations}\label{subsec:generalizingSGA}
We are now ready to \emph{combine the three F's}. We will show in Lemma \ref{lem:1-fibrationsliftfractions} that a family of arrows satisfying right fractions can be lifted through a 1-fibration of bicategories.
We obtain as a corollary that the Cartesian arrows of a fibration over a pseudocofiltered bicategory satisfy fractions, as shown in \cite[(SGA 4, Tome 2)]{GSV72} for categories.

\begin{lemma}[Pseudocofiltered implies Fractions]\label{lem:pseudocofilteredimpliesfractions}
If $\mathcal{B}$ is a pseudocofiltered bicategory as in Definition \ref{def:pseudofiltered}, then the collection of all arrows of $\mathcal{B}$ satisfies right fractions.
\end{lemma}
\begin{proof}
When $\cc{W}$ is the collection of all arrows of $\mathcal{B}$, certainly the closure properties are satisfied.
Also, \ref{0-Frc} is given by $\ref{0-pFlt}^{\textnormal{op}}$ (see Lemma \ref{lem:0-pflt-commutes-for-free}), and \ref{2-Frc} is given by $\ref{2-pFlt}^{\textnormal{op}}$.
To show \ref{1-Frc}, consider the (connected) diagram:
\[\xymatrix{
    B \ar@/^/[r]^f
        \ar@/_/[r]_g
        \ar@<.7pc>@/^1pc/[rr]^{wf}
        \ar@{}@<1.3pc>[rr]^{}="1"
        \ar@<-.7pc>@/_1pc/[rr]_{wg}
        \ar@{}@<-1.3pc>[rr]^{}="2"
    & C \ar[r]^w
    & D
    \ar@/^1pc/@{=>}^(0.7){\alpha}"1";"2"
}\]
The existence of the 2-cell $\beta\colon fu \Rightarrow gu$ such that $a_{w,g,u} \circ (\alpha \star u) = (w \star \beta) \circ a_{w,f,u}$ follows either by considering a pseudo-cone of this diagram, or as follows.
First we apply $\ref{1-pFlt}^{\textnormal{op}}$ to $f$ and $g$, we get thus an arrow $u' \colon A \to B$ and a 2-cell $\gamma\colon f u' \Rightarrow g u'$, and then we apply $\ref{2-pFlt}^{\textnormal{op}}$ to the 2-cells $a_{w,g,u'} \circ (\alpha \star u')$ and $(w \star \gamma) \circ a_{w,f,u'}$.
\end{proof}

\begin{remark}[Conditions for Fractions to imply (Pseudo)cofilteredness]\label{rem:whenfractionsimplycofiltered}
Recall that an object $1$ in a bicategory $\cc{B}$ is called (bi)terminal if it is the bilimit of the empty diagram; that is, if for all $E\colon \cc{B}$ the category $\cc{B}(E,1)$ is a contractible groupoid.
This immediately implies that for any $E\colon \cc{B}$, for any $u \colon E \to 1$ and for any $E'\colon \cc{B}$, we have
\begin{enumerate}[align=left, leftmargin=*]
    \item
    for any pair of parallel arrows $\xymatrix{
    E' \ar@/^/[r]^f
        \ar@/_/[r]_g
    & E
    }$, there is an (invertible) 2-cell $uf \simeq ug$, and
    
    \item
    for any pair of parallel 2-cells $\xymatrix{
    E' \ar@/^1pc/[r]^f_{}="1"
        \ar@/_1pc/[r]_g^{}="2"
    & E
    \ar@/^/@{=>}^{\alpha}"1";"2"
    \ar@/_/@{=>}_{\beta}"1";"2"
    }$, their whiskerings by $u$ are equal.
\end{enumerate}
Assume that $\cc{B}$ has a (bi)terminal object $1$, and let $\cc{W}$ be a family of arrows of $\cc{B}$ admitting a calculus of fractions, such that $\cc{W}$ contains all the arrows into $1$.
Note that in this case $\mathcal{B}$ is connected and hence the notions of cofilteredness and co-pseudofilteredness are equivalent via Proposition \ref{prop:connected}.
Then, using item 1 above, $\ref{1-Flt}^{\textnormal{op}}$ (for $\cc{B}$) follows from \ref{1-Frc} (for $\cc{B}$ and $\cc{W}$), and similarly \ref{2-Frc} implies $\ref{2-Flt}^{\textnormal{op}}$ using item 2.
Finally, under these hypotheses \ref{0-Frc} implies $\ref{0-Flt}^{\textnormal{op}}$ at once, so $\cc{B}$ is cofiltered.

Note that, considering the family $\cc{W}$ of all the arrows of $\cc{B}$, we have in particular that when $\cc{B}$ has a terminal object the converse of Lemma \ref{lem:pseudocofilteredimpliesfractions} holds.
\end{remark}

\begin{lemma}[Lifting Fractions Lemma]\label{lem:1-fibrationsliftfractions}
Let $P\colon \mathcal{E} \to \mathcal{B}$ be a 1-fibration of bicategories (in particular, $P$ could be a fibration).
Let $\mathcal{W}$ admit a calculus of right fractions on $\mathcal{B}$.
Then $\mathcal{C}_\mathcal{W}$ admits a calculus of right fractions on $\mathcal{E}$, where $\mathcal{C}_\mathcal{W}$ is the family of Cartesian arrows over $\mathcal{W}$.
\end{lemma}
\begin{proof} The required closure properties of Cartesian arrows are shown in Proposition \ref{prop:basicpropCartarrows}. 
In view of Proposition \ref{prop:strongequivaxiom}, we can work with the \refaxioms{0-Frc}{Frc} set of axioms instead of \refaxioms{BF}{BF3} and \refaxioms{BF}{BF4}.
Throughout this proof we denote both the arrows in $\cc{W}$ and those in $\cc{C}_{\cc{W}}$ by $\wto$ (but not all Cartesian arrows).
\begin{enumerate}[align=left,leftmargin=*]
    \item[\ref{0-Frc}]
    Consider objects $A,B,C\colon \mathcal{E}$ and arrows $w\colon A \wto B$ in $\mathcal{C_W}$ and $f\colon C \to B$.
    Use first \ref{0-Frc} in $\cc{B}$, to get a diagram of the form
    \[\xymatrix{
        D \ar@{.>}[r]^h
            \ar@{.>}[d]_u|{\circ}
            \ar@{}[dr]|{\Mreq{\alpha}}
        & PA \ar[d]^{Pw}|{\circ} \\
        PC \ar[r]_{Pf}
        & PB
    }\]
    Now choose a Cartesian lift $\hat{u}\colon \hat{D} \to C$ of $u$ (at $C$) and then consider the 2-cell $\gamma = P^{2}_{f,\hat{u}} \circ \alpha^{-1}$ and a lift of $(h,\gamma)$, as in item 0 in Lemma \ref{lem:cartesian-arrow}:
    \[
        \vcenter{\xymatrix@C+15pt@R+15pt{
            \hat{D}
                \ar@/^1em/[dr]^{f\hat{u}}
                \ar@/_2pc/@{.>}[d]_{\hat{h}}
            & \\
            A \ar[r]_{w}|{\circ}
                \ar@{}[ur]|(0.35){\Mreq{\hat{\alpha}}}
            & B
        }}
        \qquad \mapsto \qquad
        \vcenter{\xymatrix@C+15pt@R+15pt{
            D \ar[d]^(0.6){h}
                \ar@/_2pc/@{.>}[d]_{P\hat{h}}
                \ar@{{}{ }{}}@/_1pc/[d]|{\Mreq{\hat{\beta}}}
                \ar@/_1em/[dr]|(0.3){(Pf)u}
                \ar@/^1em/[dr]^(0.8){P(f\hat{u})}
                \ar@{}[dr]|(0.6){\Mreq{P^2_{f,\hat{u}}}}
            & \\
            PA \ar[r]_{Pw}|{\circ}
                \ar@{}[ur]|(0.25){\Mreq{\alpha^{-1}}}
            & PB
        }}  
    \]
    The diagram on the left then shows \ref{0-Frc} as required.

    \item[\ref{1-Frc}]
    Let us consider objects $B,C,D\colon \mathcal{E}$, arrows $w\colon C \wto D$ in $\mathcal{C_W}$, $f,g\colon B \to C$, and a 2-cell $\alpha\colon wf \Rightarrow wg$.
    We define the 2-cell $\widetilde{P\alpha}$ as the composition $PwPf \simeq P(wf) \Mr{P\alpha} P(wg) \simeq PwPg$, where the unnamed isomorphisms are   structural 2-cells of $F$.
    We proceed now in three steps which are outlined in the diagram below
    \[
        \vcenter{\xymatrix{
            B \ar@/^/[r]^f
                \ar@/_/[r]_g 
                \ar@<.7pc>@/^1pc/[rr]^{wf}
                \ar@{}@<1.3pc>[rr]^{}="1"
                \ar@<-.7pc>@/_1pc/[rr]_{wg}
                \ar@{}@<-1.3pc>[rr]^{}="2"
            & C \ar[r]^w|{\circ}
            & D
            \ar@/^1pc/@{=>}^(0.7){\alpha}"1";"2"
        }}
        \quad \stackrel{P}{\mapsto} \quad    
        \vcenter{\xymatrix{
            PB \ar@/^/[r]^{Pf}
                \ar@/_/[r]_{Pg} 
                \ar@<.7pc>@/^1pc/[rr]^{PwPf}
                \ar@{}@<1.3pc>[rr]^{}="1"
                \ar@<-.7pc>@/_1pc/[rr]_{PwPg}
                \ar@{}@<-1.3pc>[rr]^{}="2"
            & PC \ar[r]^{Pw}|{\circ}
            & PD
            \ar@/^1pc/@{=>}^(0.7){\widetilde{P\alpha}}"1";"2"
        }}
        \quad \stackrel{\bf (1)}{\leadsto} \quad
        \vcenter{\xymatrix{
             A \ar@{.>}[r]^u|{\circ} 
                \ar@{.>}@<.7pc>@/^1pc/[rr]^{(Pf)u}
                \ar@{}@<1.3pc>[rr]^{}="1"
                \ar@{.>}@<-.7pc>@/_1pc/[rr]_{(Pg)u}
                \ar@{}@<-1.3pc>[rr]^{}="2"
            & PB \ar@/^/[r]^{Pf}
                \ar@/_/[r]_{Pg}
            & PC
            \ar@/_1pc/@{:>}_(0.7){\widetilde{\beta}}"1";"2"
        }}
    \]
    \[
        \qquad \qquad \qquad \qquad \qquad \qquad
        \stackrel{\bf (2),(3)}{\leadsto}
        \quad
        \vcenter{\xymatrix{
            \hat{A} \ar@{.>}[r]^{\hat{u}}|{\circ} 
                \ar@{.>}@<.7pc>@/^1pc/[rr]^{f\hat{u}}
                \ar@{}@<1.3pc>[rr]^{}="1"
                \ar@{.>}@<-.7pc>@/_1pc/[rr]_{g\hat{u}}
                \ar@{}@<-1.3pc>[rr]^{}="2"
            & B \ar@/^/[r]^f
                \ar@/_/[r]_g
            & B
            \ar@/_1pc/@{:>}_(0.7){\widehat{\beta}}"1";"2"
        }}
    \]   
    
    In step {\bf (1)}, we use \ref{1-Frc} in $\cc{B}$.
    We get then $u$ and $\widetilde{\beta}$ such that the following equation holds:
    \begin{equation}\label{eq:tildePalpha}
        a_{Pw,Pg,u} \circ(\widetilde{P\alpha}\star u) = ((Pw)\star \widetilde{\beta}) \circ a_{Pw,PF,u}
    \end{equation}
    
    In step {\bf (2)}, we take a Cartesian lift $\hat{u}$ of $u$ (at $B$).
    
    In step {\bf (3)}, we apply item 1 in Lemma \ref{lem:cartesian-arrow} to construct $\hat{\beta}$ using: $\textnormal{``}f\textnormal{''} \coloneqq w$, $\textnormal{``}g\textnormal{''} \coloneqq f\hat{u}$, $\textnormal{``}h\textnormal{''} \coloneqq g\hat{u}$, $\textnormal{``}\alpha\textnormal{''}$ as the composition $w(f\hat{u}) \simeq (wf)\hat{u} \Mr{\alpha\star \hat{u}} (wg)\hat{u} \simeq w(g\hat{u})$ and $\textnormal{``}\beta\textnormal{''}$ as the composition $P(f\hat{u}) \simeq (Pf)u \Mr{\widetilde{\beta}} (Pg) u \simeq P(g\hat{u})$ (where the unnamed isomorphisms are either  structural 2-cells of $F$ or associators).
    To verify the hypothesis in the Lemma amounts to checking that $P(\alpha \star \hat{u})$ equals the following composition (where for convenience we write $\circ$ for the composition of arrows in $\cc{B}$ but we omit this symbol for the composition in $\cc{E}$)
    \[
        P((wf)\hat{u})
        \simeq
        Pw\circ (Pf\circ u)
        \Mr{(Pw)\star\widetilde{\beta}}
        Pw\circ (Pg\circ u)
        \simeq
        P((wg)\hat{u}),
    \]
    where each of the two unnamed isomorphisms are given by the structural 2-cells of $F$ and the associators, or equivalently that the two dashed paths in the cube below are equal.
    \[\begin{tikzcd}[execute at end picture={
            \begin{scope}[on background layer]
                \draw[->] (a) -- (b) node[below,xshift={-4em}] {$\scriptstyle {a_{Pw,Pf,u}}$};
            \end{scope}
            \draw[preaction={draw,line width=1pt,white},dashed,->] (c) -- (d)
                node[above,xshift={-4em}] {$\scriptstyle {Pa_{w,g,\hat{u}}}$};
        }]
    	{P((wf)\hat{u})}
    	&& {P(w(f\hat{u}))} \\
    	{P(wf)\circ u}
    	& |[alias = c]| {P((wg)\hat{u})}
    	& {Pw\circ P(f\hat{u})}
    	&  |[alias = d]| {P(w(g\hat{u}))} \\
    	|[alias = a]| {(Pw\circ Pf) \circ u}
    	& \contour{white}{$P(wg)\circ u$}
    	& |[alias = b]| {Pw\circ (Pf \circ u)}
    	& {Pw\circ P(g\hat{u})} \\
    	& {(Pw\circ Pg) \circ u}
    	&& {Pw\circ(Pg\circ u)}
    	\arrow["{\widetilde{P\alpha} \star u}"{description}, from=3-1, to=4-2]
    	\arrow["{P\alpha \star u}"{description}, from=2-1, to=3-2]
    	\arrow[dashed, "{P(\alpha\star \hat{u})}"{description}, from=1-1, to=2-2]
    	\arrow["{P^{-2}_{wf,\hat{u}}}"', from=1-1, to=2-1]
    	\arrow["{P^{-2}_{wf,\hat{u}}}", from=2-2, to=3-2]
    	\arrow[dashed,"{P^{-2}_{w,f\hat{u}}}", from=1-3, to=2-3]
    	\arrow[dashed,"{P^{-2}_{w,f\hat{u}}}", from=2-4, to=3-4]
    	\arrow[dashed,"{(Pw)\star P^{-2}_{f,\hat{u}}}", from=3-4, to=4-4]
    	\arrow[dashed,"{Pa_{w,f,\hat{u}}}", from=1-1, to=1-3]
    	\arrow["{P^{-2}_{w,f}\star u}"', from=2-1, to=3-1]
    	\arrow["{P^{-2}_{w,f}\star u}", from=3-2, to=4-2]
    	\arrow["{a_{Pw,Pg,u}}"', from=4-2, to=4-4]
    	\arrow[dashed,"{(Pw)\star P^{-2}_{f,\hat{u}}}", from=2-3, to=3-3]
    	\arrow[dashed,"{Pw \star \widetilde{\beta}}"{description}, from=3-3, to=4-4]
    \end{tikzcd}\]
    
    Since the top part of the left face, the front and the back faces are all commutative by the definition of pseudo-functor, the bottom part of the left side is commutative by the definition of $\widetilde{P\alpha}$ above, and the commutativity of the bottom face is precisely equation \eqref{eq:tildePalpha}, these two paths are indeed equal as desired.
    
    The equation $\textnormal{``}f \star \widehat{\beta} = \alpha\textnormal{''}$ in the conclusion of item 1 in Lemma \ref{lem:cartesian-arrow}, when applied with the definitions above, becomes then $a_{w,g,\hat{u}} \circ (\alpha \star \hat{u}) = (w \star \widehat{\beta}) \circ a_{w,f,\hat{u}}$, as required in the axiom \ref{1-Frc}.

    \item[\ref{2-Frc}]
    Let us consider objects $B,C,D\colon\mathcal{E}$, arrows $w\colon C \wto D$ in $\cc{C}_\mathcal{W}$, $f,g\colon B \to C$, and 2-cells $\alpha, \beta\colon f \Rightarrow g$, such that $w \star \alpha = w \star \beta$.
    We apply $P$ and use \ref{2-Frc} in $\cc{B}$, in this way we have $u\colon A \wto PB$ in $\cc{W}$ such that $(P\alpha) \star u = (P\beta) \star u$, and we lift $u$ to a Cartesian arrow $\hat{u}\colon \hat{A} \to B$.
    Then, since by pre- and post-composing by structural cells of $P$ we get $P(\alpha \star \hat{u}) = P(\beta \star \hat{u})$ and since $w \star (\alpha \star \hat{u}) = w \star (\beta \star \hat{u})$, by item 2 in Lemma \ref{lem:cartesian-arrow} $\alpha \star \hat{u} = \beta \star \hat{u}$ and we are done.
\end{enumerate}
\end{proof}

\noindent
Combining Lemmas \ref{lem:pseudocofilteredimpliesfractions} and \ref{lem:1-fibrationsliftfractions}, we have a bicategorical version of 
\cite[(SGA 4, Tome 2) exp. VI, Prop. 6.4]{GSV72}:

\begin{corollary}\label{coro:BpseudofilteredEfractionsasinSGA}
Let $\mathcal{E} \to \mathcal{B}$ be a fibration of bicategories. If $\mathcal{B}$ is pseudocofiltered, then the Cartesian arrows satisfy right fractions.
\end{corollary}

Of course, we also have the dual results.
Dualizing a calculus of fractions by $-^{co}$ doesn't change the definition, and dualizing a right/left calculus of fractions by $-^{\textnormal{op}}$ makes it into a left/right calculus of fractions. Hence a co-fibration (see \cite[Remark 2.2.14]{Buc14}) lifts left calculi, an op-fibration lifts right calculi, and a coop-fibration lifts left calculi. 
We record here the following version of Corollary \ref{coro:BpseudofilteredEfractionsasinSGA}, which is the one we will actually use to compute colimits in this paper.

\begin{corollary} \label{coro:BpseudofilteredEfractions}
Let $\mathcal{E} \to \mathcal{B}$ be a co-fibration of bicategories. If $\mathcal{B}$ is pseudofiltered, then the co-Cartesian arrows satisfy left fractions.
\end{corollary}

\section{Bicategory-Indexed Tricolimits of Bicategories}\label{sec:colimits}
In ordinary 1-category theory, a pseudo-colimit of categories can be computed by localizing, at the Cartesian arrows, the fibration associated to the diagram by its Grothendieck construction \cite[(SGA 4, Tome 2) Expos\'e VI Section 6]{GSV72}.
In this section we show, using fibrations and localizations of bicategories as described in Sections \ref{subsec:fibrations} and \ref{sec:diagrammatic} respectively, how conical tricolimits of bicategories (as in Definition \ref{def:tricolimit}) can similarly be computed by localizing the associated fibration at the Cartesian arrows and 2-cells.
As an application, we show in Section \ref{subsec:colimitcat} that we can compute bicategory-indexed pseudofiltered pseudo-colimits of categories by using only the ordinary one-dimensional calculus of fractions.

\subsection{Computing colimits in Bicat}\label{subsec:computingcolimits}
Let $\mathcal{B}$ be a bicategory.
In this subsection we will consider a trihomomorphism $F\colon \mathcal{B} \to \Bicat$, as originally defined in \cite{GPS95} and developed in more detail in \cite{Gur09}, and show that one can compute its conical tricolimit by localizing the bicategory given by its Grothendieck construction (or bicategory of elements, $\el F$).

We recall that for arbitrary tricategories $\cc{A}$ and $\cc{B}$, there is a tricategory $[\cc{B},\cc{A}]$ of trihomomorphisms, trinatural transformations, trimodifications and perturbations, defined in \cite{GPS95,Gur09}.
A trihomomorphism $F$ as above can be seen as an object of $[\cc{B},\Bicat]$, when $\cc{B}$ is interpreted as a tricategory with trivial 3-cells.

\begin{remark}[Variance and Duality]\label{rem:variance-duality} 
Since we have chosen to work with a covariant $F\colon \mathcal{B} \to \Bicat$, it is convenient for us to consider the bicategory of elements $\el F$, with objects given by pairs $(C,x)$, with $x \colon FC$ as usual, (that we will denote here as $(x,x_-)$ for the sake of comparison with \cite{Buc14}) and the following arrows and 2-cells
\[\begin{tabular}{c|c}
    $(f,f_-)\colon (x,x_-) \to (y,y_-)$
    & $(\alpha,\alpha_-)\colon (f,f_-) \Rightarrow (g,g_-)$\\
    $\xymatrix{
        Ff(x_-) \ar[r]^{f_-}
        & y_-
    }$
        & $\xymatrix{
            Ff(x_-) \ar[dr]_{(F\alpha)_{x_-}}
                \ar[rr]^{f_-}
                \ar@{}[drr]|{\alpha_- \Downarrow}
            && y_- \ar@{<-}[dl]^{g_-} \\
            & Fg(x_-)
            &
        }$
\end{tabular}\]
Throughout the paper, $\el F$ will refer to this bicategory, that can be traced back to at least \cite{Str76} for the case of 2-categories. 
We found this particular choice of directions to be convenient for doing computations in the covariant case, since it induces a covariant trihomomorphism $\el \colon [\mathcal{B},\Bicat] \to \Bicat/\mathcal{B}$, without any appearance of dual bicategories.
We remark however that this is done only for convenience, and we could get similar results in this paper working with the construction $\el'$ below instead (this is similar to how one can think of a pseudocone as either a lax or oplax cone whose structural 2-cells are invertible).
    
The bicategory $\el F$ can be compared with the construction in \cite{Bak} and \cite{Buc14}, where (following the classical correspondence between fibrations and pseudo-functors) one starts from a trihomomorphism $\mathcal{B}^{coop} \to \Bicat$ and defines the arrows and 2-cells of its Grothendieck construction as:
\[\begin{tabular}{c|c}
    $(f,f_-)\colon (x,x_-) \to (y,y_-)$
        & $(\alpha,\alpha_-)\colon (f,f_-) \Rightarrow (g,g_-)$\\
    $\xymatrix{
        x_- \ar[r]^{f_-}
        & Ff(y_-)
    }$
        & $\xymatrix{
            x_- \ar[dr]_{g_-}
                \ar[rr]^{f_-}
                \ar@{}[drr]|{\alpha_- \Downarrow}
            && Ff(y_-) \ar@{<-}[dl]^{(F\alpha)_{y_-}} \\
            & Fg(y_-)
            &
        }$
\end{tabular}\]
If we denote this construction by $\el'$, we observe that we have
\begin{equation}\label{eq:elasel'}
    \el F \coloneqq (\el' ((D_{\textnormal{op}} \circ F)^{co}))^{\textnormal{op}}
\end{equation}
where $D_{\textnormal{op}}\colon \Bicat \to \Bicat^{co}$ is the dual operator sending a bicategory to its op-dual bicategory $\mathcal{X} \mapsto \mathcal{X}^{op}$, and where the $\el '$ construction is taken over the base bicategory $\mathcal{B}^{op}$ instead of over $\mathcal{B}$.
As it is defined as a $(\textnormal{fibration})^{op}$, we refer to it (and to the bicategories having the corresponding lifting properties) as a {co}-fibration.

Using \eqref{eq:elasel'} we can translate definitions and results from \cite{Buc14} to our setting, as in the following proposition.
\end{remark}

\sloppy
\begin{proposition}[{\cite[Prop.~3.3.4]{Buc14}}] \label{prop:grothendieckyieldscofibration}
For each trihomomorphism $F\colon\mathcal{B} \to \Bicat$, its Grothendieck construction yields a {co}-fibration of bicategories $P_F\colon\el F \to \mathcal{B}$ whose co-Cartesian arrows (resp. 2-cells) are those whose second coordinate is an equivalence (resp. invertible).
\qed
\end{proposition}
\fussy

\begin{definition}[Constant Trihomomorphism]\label{def:constant}
Let $\cc{B}, \cc{A}$ be tricategories.
The trihomomorphism $\Delta\colon\cc{A} \to [\cc{B},\cc{A}]$ maps $A\colon\cc{A}$ to the constant trihomomorphism that maps all objects of $\mathcal{B}$ to $A$, and all arrows, 2- and 3-cells of $\cc{B}$ to identities.
\end{definition}

Consider the Cartesian product $\mathcal{B} \times \cc{X}$ of bicategories, whose structure is defined pairwise.
Then the projection $\pi_1\colon \mathcal{B} \times \cc{X} \to \cc{B}$ is a fibration of bicategories whose Cartesian arrows (resp. 2-cells) are those pairs of arrows (resp. 2-cells) whose second coordinate is an equivalence (resp. invertible) in $\cc{X}$. 
The following results follows from \cite[Constr. 3.3.3]{Buc14}.

\begin{remark}[Constant Trihomomorphism]\label{rem:clearonobjects}
For any bicategory $\cc{X}$, there is a biequivalence of bicategories (that is actually an isomorphism) $\el (\Delta \cc{X}) \cong \mathcal{B} \times \cc{X}$, making the following triangle commute strictly
\[\xymatrix{
    \el (\Delta \cc{X}) \ar[rd]_{P_{\Delta\cc{X}}}
        \ar[rr]^\cong
    && \mathcal{B} \times \cc{X} \ar[dl]^{\pi_1} \\
    & \cc{B}
}\]
\qed
\end{remark}

We will now consider the following \emph{strict slice} tricategory $\Bicat/\mathcal{B}$, in which Buckley's generalisation of the Grothendieck construction naturally lands:

\begin{enumerate}\setcounter{enumi}{-1}
    \item
    its objects are pairs $(\cc{E},P)$ where $\cc{E}$ is a bicategory and $P\colon \cc{E}\to \mathcal{B}$ is a pseudo-functor - when it is clear we omit the pseudo-functor and denote these by $\cc{E}$,
    
    \item
    its arrows $F\colon (\cc{E},P) \to (\cc{E}^{\prime},P^{\prime})$ are pseudo-functors $F\colon \cc{E} \to \cc{E}^{\prime}$ such that $P^{\prime}\circ F = P$,
    
    \item
    its 2-cells $\alpha\colon F \to F^{\prime}$ are pseudo-natural transformation such that $P^{\prime} \star \alpha = 1_P$,
    
    \item
    its 3-cells $\Gamma\colon \alpha \to \alpha^{\prime}$ are modifications such that $P^{\prime} \star \Gamma = 1_{1_P}$ (where $\star$ stands for the whiskering of a pseudo-functor and a modification).
\end{enumerate}

\begin{remark}[Product of Bicategories]\label{rem:product}
For any bicategory $\cc{A}$, and any pseudo-functor $P\colon \cc{E} \to \cc{B}$, there is a {natural} biequivalence of bicategories (that is actually an isomorphism)
\begin{equation}\label{eq:cartprodandslice}
    (\Bicat/\mathcal{B}) ( \cc{E}, \mathcal{B} \times \cc{A}) \cong [\cc{E},\cc{A}]
\end{equation}
\end{remark}

The tricategory $\textrm{Fib}(\mathcal{B})$ is defined  in \cite[Def. 3.2.4]{Buc14} as a sub-tricategory of $\Bicat/\mathcal{B}$, full on 2- and 3-cells, whose objects are those pseudo-functors that are fibrations of bicategories, and whose arrows are those pseudo-functors that are Cartesian as defined in \cite[Def. 3.2.4]{Buc14}, that is such that they respect both Cartesian arrows and 2-cells.
We define $\coFib(\mathcal{B})$ analogously.
The biequivalence of bicategories considered in Remark \ref{rem:clearonobjects} is clearly Cartesian in this sense. 
Recalling then the notation in Definition \ref{def:localizationformodel}, we conclude:

\begin{remark}[Constant Trihomomorphism]\label{rem:constant}
For any bicategory $\cc{A}$, and any pseudo-functor $P\colon \cc{E} \to \cc{B}$, there is a {natural} biequivalence of bicategories (that is actually an isomorphism)
\begin{equation}\label{eq:combinetwo}
    \coFib(\cc{B})(\cc{E}, \el (\Delta \cc{A})) \cong [\cc{E},\cc{A}]_{\cc{C}_{1,2}}
\end{equation}
(given by postcomposition with the biequivalence of bicategories  considered in Remark \ref{rem:clearonobjects} together with \eqref{eq:cartprodandslice})
where $[\cc{E},\cc{A}]_{\cc{C}_{1,2}}$ is the bicategory introduced in Definition (\ref{def:localizationformodel}) and $\cc{C}_1$, resp. $\cc{C}_2$, are the families of co-Cartesian arrows, resp. 2-cells, of $\cc{E}$.
\end{remark}

\begin{definition}[Tricolimits]\label{def:tricolimit}
Let $F \colon\cc{B} \to \cc{A}$ be a trihomomorphism between tricategories. 
We consider the trihomomorphism
\[
    \cc{A}
    \mr{\Delta -}
    [\cc{B},\cc{A}]
    \xr{[\cc{B},\cc{A}](F,-)}
    \Bicat
\]
mapping $A\colon \cc{A}$ to the bicategory $[\cc{B},\cc{A}](F,\Delta A)$.
If this trihomomorphism is representable (in the sense of tricategory theory), we say that $F$ has a tricolimit, and we refer to the object $L$ representing it as the tricolimit of $F$.
Explicitly, this amounts to saying that there are biequivalences of bicategories, natural in $A$:
\[
    [\cc{B},\cc{A}](F,\Delta A) \simeq \cc{A}(L,A)
\]
\end{definition}

\begin{proposition}[{\cite[Prop.~3.3.12]{Buc14}}, Local Biequivalence]\label{prop:grothendiecklocalequivalence}
For $F,G\colon \mathcal{B} \to \Bicat$, the Grothendieck construction yields a biequivalence of bicategories
\begin{equation}\label{eq:grothendiecklocalequivalence}
    [\mathcal{B},\Bicat](F, G)
    \simeq
    \coFib(\mathcal{B})(\el F,\el G)
\end{equation}
that is natural in $F$ and $G$.
\qed
\end{proposition}

We are now ready to prove:

\begin{theorem}[Conical Tricolimits in $\Bicat$]\label{thm:tricolimit}
Let $\mathcal{B}$ be a bicategory, and $F\colon \mathcal{B} \to \Bicat$ be a trihomomorphism. 
Then the tricolimit of $F$ in $\Bicat$ is given by the localization $(\el F)[\mathcal{C}_{1,2}^{-1}]$ as defined in Definition \ref{def:localizationformodel}, where $\cc{C}_1$, resp. $\cc{C}_2$ are the families of co-Cartesian arrows, resp. co-Cartesian 2-cells of $\el F$.
\end{theorem}
\begin{proof} 
Using in turn \eqref{eq:grothendiecklocalequivalence} and \eqref{eq:combinetwo}, we have the desired natural biequivalence of bicategories
\[
    [\mathcal{B},\Bicat](F, \Delta(-))
    \simeq
    \coFib(\mathcal{B})(\el F,\el (\Delta(-)))
    \simeq
    [\el F, -]_{\cc{C}_{1,2}}
\]
(Note that the tricolimit of $F$, resp. the desired localization, is defined as a trirepresentation of the trihomomorphism on the left hand side, resp. right hand side.)
\end{proof}

As far as we know, the construction of a localization of a bicategory at both arrows and non-trivial 2-cells has never been considered. 
In this paper, we will apply this result in a case when we won't need to localize by the 2-cells: the \emph{discrete case}, that is the computation of pseudofiltered pseudo-colimits of categories in the following section. 
As will be discussed in a follow-up paper \cite{BPS}, pseudofiltered tricolimits of bicategories can also be computed without localizing at the 2-cells, because the condition in the following corollary holds when $\cc{B}$ is pseudofiltered:

\begin{corollary}\label{coro:arrows-sufficient}
Let $\mathcal{B}$ be a bicategory, and $F\colon \mathcal{B} \to \Bicat$ be a trihomomorphism.
Suppose that the localization $(\el F)[\mathcal{C}_{1}^{-1}]$ exists and that $L_{\mathcal{C}_{1}}\colon \el F \to (\el F)[\mathcal{C}_{1}^{-1}]$ sends co-Cartesian 2-cells to invertible ones, where $\cc{C}_1$ is the families of co-Cartesian arrows of $\el F$.
Then the tricolimit of $F$ in $\Bicat$ exists and is given by the localization $(\el F)[\mathcal{C}_{1}^{-1}]$. 
\end{corollary}
\begin{proof}
For $\cc{D}$ a bicategory, as all pseudo-functors in $[\el F,\cc{D}]_{\cc{C}_1}$ factor through $L_{\mathcal{C}_{1}}$ up to natural equivalence, they all send co-Cartesian 2-cells to invertible ones.
Hence $[\el F,\cc{D}]_{\cc{C}_1} \subseteq [\el F,\cc{D}]_{\cc{C}_{1,2}}$. We also have by definition $[\el F,\cc{D}]_{\cc{C}_{1,2}} \subseteq [\el F,\cc{D}]_{\cc{C}_{1}}$.
This implies that both universal properties coincide: a localization by $\cc{C}_1$ is equivalently a localization by $\cc{C}_{1,2}$.
The result is then a direct consequence of Theorem \ref{thm:tricolimit} above.
\end{proof}

\subsection{Pseudofiltered pseudo-colimits in Cat}\label{subsec:colimitcat}
We consider now a pseudo-functor $F\colon \cc{B} \to \Cat$ from $\cc{B}$ a bicategory, and view it as a discrete trihomomorphism $dF\colon \cc{B} \to \Bicat$.
Looking at the proof of Theorem \ref{thm:tricolimit} in this case, we notice that we have equivalences of categories (natural in ${\bf E}\colon\Cat$)
\[
    [\mathcal{B},\Cat](F, \Delta {\bf E})
    \cong
    [\mathcal{B},\Bicat](dF, \Delta d {\bf E})
    \simeq
    \coFib(\mathcal{B})(\el (dF),\el (\Delta d{\bf E}))
    \cong
    [\el (dF), d {\bf E}]_{\cc{C}_1,\cc{C}_2}
\]

By looking explicitly at the proof of the local equivalence of the Grothendieck construction in \cite[Prop. 3.3.12]{Buc14}, we can observe that the only mere equivalence ($\simeq$) left in this chain is, like the others, actually an isomorphism ($\cong$).
Indeed, in the discrete case it is immediate that the inverse Grothendieck construction is an actual inverse, rather than only a quasi-inverse.
Noting that ${\bf E}$ has only trivial 2-cells, we can continue this chain of natural isomorphisms as follows:
\[
    [\el (dF), d {\bf E}]_{\cc{C}_{1,2}}
    \cong
    [\el (dF), d {\bf E}]_{\cc{C}_1}
    \cong
    [\pi_0 \el (dF),{\bf E}]_{[\cc{C}_1]}
    \cong
    [(\pi_0 \el (dF))[[\cc{C}_1]^{-1}], {\bf E}].
\]
where $[\mathcal{C}_1]$ is the family of equivalence classes of co-Cartesian arrows after applying $\pi_0$, as in Lemma \ref{lem:pi0}.
Noting that a pseudo-colimit of a pseudo-functor between bicategories is defined as a strict representation of the pseudo-functor $[\mathcal{B},\Cat](F, \Delta {\bf E})$, we conclude:

\begin{corollary}[Discrete Case]\label{coro:colimit-cat}
Let $F\colon \mathcal{B} \to \Cat$ be a pseudo-functor. 
Then the pseudo-colimit of $F$ is given by the category $\pi_0 (\el d F)[\cc{W}^{-1}]$, where $\cc{W} =[\cc{C}_1]$ is the family of equivalence classes of co-Cartesian arrows.
\qed
\end{corollary}

\begin{remark}
Note that when $\mathcal{B}$ is a pseudofiltered bicategory, by Proposition \ref{prop:grothendieckyieldscofibration}, Corollary \ref{coro:BpseudofilteredEfractions} and Lemma \ref{lem:pi0}, the category $(\pi_0 \el d F)[\cc{W}^{-1}]$ can be computed as a category of fractions.
\end{remark}

We fix a pseudofiltered bicategory $\mathcal{B}$ and a pseudo-functor $F\colon \mathcal{B} \to \Cat$.
We will now give an explicit formula of its pseudo-colimit, by describing the category $(\pi_0 \el d F)[\cc{W}^{-1}]$. 
An idea to have in mind, in order to understand the content of the following proposition and its proof, is that the computation of this category as a category of fractions has redundant information, and the presentation of Proposition \ref{prop:filteredcolimit} is what we get by discarding this redundancy.
When $\cc{B}$ is a (strict) 2-category, and $F$ is a (strict) 2-functor, this construction can immediately be seen to match the one in \cite{DS06}.

\begin{proposition}\label{prop:filteredcolimit}
Let $\mathcal{B}$ be a pseudofiltered bicategory and $F\colon \mathcal{B} \to \Cat$ a pseudo-functor.
The pseudo-colimit of $F$ can be constructed as the category $\Colim_{\mathcal{B}}F$ defined by the following data:
\begin{enumerate}[align=left,leftmargin=*]
    \item[\textbf{Objects:}]
    Pairs $(A,x)$ where $A\colon\mathcal{B}$ is an object and $x\colon FA$ is an object of $FA$.
    
    \item[\textbf{Premorphisms:}]
    Quadruples $(C, u, v, \xi)\colon(A,x) \to (B,y)$ where $C\colon\mathcal{B}$ is an object, $u\colon A \to C$ and $v\colon B \to C$ are arrows and $\xi\colon Fu(x) \to Fv(y)$ is an arrow of $FC$.
    
    \item[\textbf{Homotopies:}]
    Quintuples $(C, w_1, w_2, \alpha, \beta)\colon (C_1, u_1, v_1, \xi_1) \equiv (C_2, u_2, v_2, \xi_2)$, where $C\colon\mathcal{B}$ is an object, $w_1\colon C_1 \to C$ and $w_2\colon C_2 \to C$ are arrows and $\alpha\colon w_1u_1 \simeq w_2u_2$ and $\beta\colon w_1v_1 \simeq w_2v_2$ are invertible 2-cells such that
    \begin{equation}\label{eq:LL}\vcenter{\xymatrix{
        Fw_1\circ Fu_1 (x) \ar[r]^{(F^2_{w_1, u_1})_x}
            \ar[d]_{Fw_1(\xi_1)}
        & F(w_1u_1)(x) \ar[r]^{(F\alpha)_x}
        & F(w_2u_2)(x) \ar[r]^{(F^{-2}_{w_2, u_2})_x}
        & Fw_2\circ Fu_2(x) \ar[d]^{Fw_2(\xi_2)} \\
        Fw_1\circ Fv_1 (y) \ar[r]_{(F^2_{w_1, v_1})_y}
        & F(w_1v_1)(y) \ar[r]_{(F\beta)_y}
        & F(w_2v_2)(y) \ar[r]_{(F^{-2}_{w_1, v_2})_y}
        & Fw_2\circ Fv_2(y)
    }}\end{equation}

    \item[\textbf{Arrows:}]
    Equivalence classes of premorphisms under the homotopy relation in which two premorphisms are said to be homotopic if there is a homotopy between them.
    
    \item[\textbf{Identities:}]
    For $(A,x)$ an object we define $1_{(A,x)} = [(A, 1_A, 1_A, 1_{F1_A(x)})]$.
    
    \item[\textbf{Composition:}]
    For $\xymatrix{
    (A,x) \ar[rr]^{[(C_1, u_1, v_1, \xi_1)]}
    && (B,y) \ar[rr]^{[(C_2, u_2, v_2, \xi_2)]}
    && (C,z)
    }$ we define the composite as follows.
    For any $D\colon\mathcal{B}$ and $f\colon C_1 \to D$, $g\colon C_2 \to D$ such that we have an invertible 2-cell $\gamma\colon fv_1 \simeq gu_2$, which can be constructed using for example \ref{0-pFlt} (and Lemma \ref{lem:0-pflt-commutes-for-free}), we define 
    \begin{multline}\label{eq:compositioninthepseudocolim}
        [(C_2, u_2, v_2, \xi_2)] \circ [(C_1, u_1, v_1, \xi_1)] =\\
        [(D,fu_1, gv_2, (F^{2}_{g,v_2})_z\circ Fg(\xi_2)\circ (F^{-2}_{g,u_2})_y\circ (F\gamma)_y\circ (F^{2}_{f,v_1})_y\circ Ff(\xi_1)\circ (F^{-2}_{f,u_1})_x)]
    \end{multline}
\end{enumerate}
\end{proposition}
\begin{proof}
The category $(\pi_0 \el d F)[\cc{W}^{-1}]$, which we know from Corollary \ref{coro:colimit-cat} to be the pseudo-colimit of $F$, can be rather complicated to describe.
Indeed, two quotient-like constructions interfere: an identification of arrows ($\pi_0$) and a free inversion of arrows (localization by calculus of fractions).
In order to sort out this interplay, we introduce the above category $\Colim F$ which can be seen as a ``sub-category'' of $(\pi_0 \el d F)[\cc{W}^{-1}]$: it has the same objects, but \emph{a priori} fewer arrows and \emph{a priori} fewer homotopies between these arrows.
The inclusion of $\Colim F \hookrightarrow (\pi_0 \el d F)[\cc{W}^{-1}]$ is denoted by $K$ in the proof.
The notion of identities and composition chosen for $\Colim F$ are sent to the corresponding notions in $(\pi_0 \el d F)[\cc{W}^{-1}]$ by this inclusion (as detailed in Remark \ref{rem:oncompositioninthepseudocolim}).
As we later prove that existence of homotopies in $(\pi_0 \el d F)[\cc{W}^{-1}]$ implies existence of homotopies in $\Colim F$, we get for free that  $\Colim F$ as described in the proposition is a well-defined category.
The proof is then structured as follows: we first unfold the definition of $(\pi_0 \el d F)[\cc{W}^{-1}]$, then we construct the functors $H: \Colim F \leftrightarrow (\pi_0 \el d F)[\cc{W}^{-1}]\colon K$ and finally we prove that the pair forms an isomorphism of categories.

We first construct the category $(\pi_0 \el d F)[\cc{W}^{-1}]$ as a category of fractions.
Its objects are then pairs $(A,x)$, just as in the Proposition.
Its arrows $(A,x) \to (B,y)$ are equivalence classes of roofs as follows:

\begin{enumerate}[align=left,leftmargin=*]
    \item[\textbf{Roof:}] 
    $\vcenter{\xymatrix@R=.5em{
    & (C,z) \ar@{<-}[rd]^{(v,\nu)}|{\circ}
        \ar@{<-}[ld]_{(u,\mu)} \\
    (A,x)
    && (B,y)
    }}$ with $C,u,v$ as in the Proposition, $z\colon FC$,  $\mu\colon Fu(x) \to z$ an arrow of $FC$, and $\nu\colon Fv(y) \simeq z$  an isomorphism of $FC$.
    
    \item[\textbf{Equivalence of roofs:}]
    Two roofs are in the same class, as usual, if they can be made part of a third common roof.
    That is, if there exists a cospan between their ``ridges''
    \[\xymatrix@R=1.5em{
        && (C',z') \ar@{<-}[dr]^{(w_2',\lambda)}
            \ar@{<-}[dl]_{(w_1',\kappa)} \\
        & (C_1,z_1) \ar@{<-}[drrr]_{(v_1,\nu_1)}|{\circ}
            \ar@{<-}[dl]_{(u_1,\mu_1)}
        && (C_2,z_2) \ar@{<-}[dr]^{(v_2,\nu_2)}|{\circ}
            \ar@{<-}[dlll]^{(u_2,\mu_2)}|!{[ll];[dr]}\hole \\
        (A,x)
        &&&& (B,y)
    }\]
    in $\el d F$ such that the equivalence class of the right leg is in $\mathcal{W}$ and such that these two equations hold in $\pi_0 \el d F$:
    \begin{equation}\label{eq:twoeqinpi0}
        [(w_1', \kappa)(u_1,\mu_1)]
        =
        [(w_2', \lambda)(u_2,\mu_2)]
        \qquad\qquad
        [(w_1', \kappa)(v_1,\nu_1)]
        =
        [(w_2', \lambda)(v_2,\nu_2)]
        \in \mathcal{W}.
    \end{equation}
\end{enumerate}
\noindent
{\bf Construction of $H$.} We construct an assignement $H\colon (\pi_0 \el d F)[\cc{W}^{-1}] \to \Colim_{\mathcal{B}}F$ on objects and arrows, which is the identity on objects, and is defined on the arrows as follows. 
Note that starting with a roof $((C,z),(u,\mu),(v,\nu))$ we can construct a premorphism $(C,u,v,\xi)$ by defining $\xi = \nu^{-1} \mu$.
We will show now that two equivalent roofs yield homotopic premorphisms (with the notion of homotopy defined in the statement of the proposition), so that $H$ is well defined.

Let's describe explicitly what \eqref{eq:twoeqinpi0} means.
By the first of these two equations, in $\el d F$ we have arrows $(h_0,\sigma_0),\dots, (h_n,\sigma_n)\colon (A,x) \to (C',z')$ such that $(h_0, \sigma_0) = (w_1',\kappa)(u_1,\mu_1)$ and $(h_n, \sigma_n) = (w_2',\lambda)(u_2,\mu_2)$, and 2-cells $\{\alpha_i\}$ linking those arrows together in a ``zig-zag" of length $n$.
Since these are 2-cells in $\el d F$, we have a commutative diagram in $FC'$ given on the left of \eqref{zig-zag-in-C-prime} below (the orientation of the arrows $\{\alpha_i\}$ is not fixed, so their direction in the diagram is arbitrary). 

Looking at the second equation in \eqref{eq:twoeqinpi0}, by definition, we have an arrow $(k,\tau)$ of $\cc{W}$ such that $[(w_1',\kappa)(v_1,\nu_1)] = [(k,\tau)] = [(w_2', \lambda)(v_2,\nu_2)]$, hence with $\tau$ an invertible arrow of $FC'$.
These two equalities lead to two zig-zags of 2-cells as above, and by concatenating them we have arrows $(k_0,\tau_0),\dots, (k_m,\tau_m)$ such that $(k_0,\tau_0) = (w_1',\kappa)(v_1,\nu_1)$ and $(k_m,\tau_m) = (w_2',\lambda)(v_2,\nu_2)$, and 2-cells $\{\beta_j\}$ linking them such that
the diagram on the right of \eqref{zig-zag-in-C-prime} below is commutative, with the further assumption that one of the $\tau_j$ is invertible.
\begin{equation}\label{zig-zag-in-C-prime}
    \vcenter{\xymatrix{
        F(w_1'u_1)(x) \ar[ddrrr]^{\sigma_0}
            \ar[d]_{(F\alpha_1)_x} \\
        F(h_1)(x) \ar[drrr]^{\sigma_1} \\
        F(h_2)(x) \ar[rrr]^{\sigma_2}
            \ar[u]^{(F\alpha_2)_x}
        &&& z \\
        \vdots \ar@{<-}[u]^{(F\alpha_3)_x}
            \ar[d]_{(F\alpha_n)_x}
            \ar@{}[urrr]|(0.4){\vdots} \\
        F(w_2'u_2)(x) \ar[uurrr]_{\sigma_n}
    }}
    \qquad\qquad
    \vcenter{\xymatrix{
        &&& F(w_1'u_1)(y) \ar[ddlll]_{\tau_0}
            \ar@{<-}[d]^{(F\beta_1)_y} \\
        &&& F(k_1)(y) \ar[dlll]_{\tau_1}
            \ar[d]^{(F\beta_2)_y} \\
        z
        &&& F(k_2)(y) \ar[lll]_{\tau_2} \\
        &&& \vdots \ar[u]_{(F\beta_3)_y}
            \ar[d]^{(F\beta_m)_y}
            \ar@{.>}[ulll]|{\simeq}^(0.37){\scriptscriptstyle\vdots}_(0.4){\scriptscriptstyle\vdots} \\
        &&& F(w_2'u_2)(y) \ar[uulll]^{\tau_m}
    }}
\end{equation}

We now consider all the connected data given by the 2-cells $\{\alpha_i\}$, $\{\beta_j\}$ in the pseudofiltered bicategory $\mathcal{B}$, as on the left of the diagram below, and we take a pseudo-cocone $(C,\phi)$ on it.
We will define a homotopy using the data of the pseudo-cocone.
More precisely (see the diagram below on the right), we define $w_1 = \phi_{C_1}$, $w_2 = \phi_{C_2}$, $\alpha = \phi_{u_2}^{-1}\phi_{u_1}\colon w_1u_1 \simeq w_2u_2\colon A \to C$ and $\beta = \phi_{v_2}^{-1}\phi_{v_1}\colon w_1v_1 \simeq w_2v_2\colon B \to C$
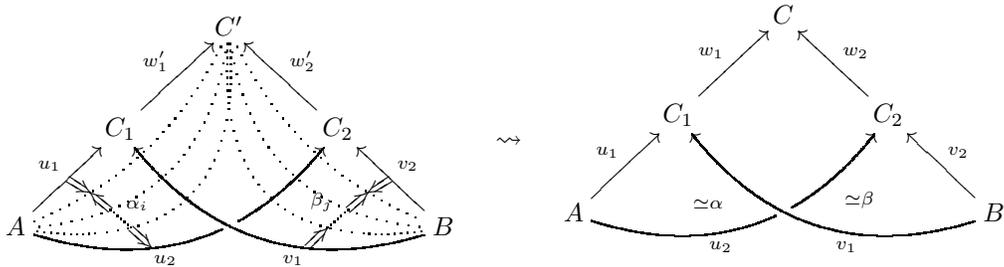
\[
    \vcenter{\xymatrix{
        && C^{\prime} \\
        & C_1 \ar[ur]^{w_1^{\prime}}
        && C_2 \ar[ul]_{w_2^{\prime}} \\
        A \ar[ur]^{u_1}|{}="1"
            \ar@/_2pc/[urrr]|(0.61)\hole_(0.4){u_2}="5"
            \ar@{.}@/_1pc/[uurr]|(0.27){}="2"
            \ar@{.}@/_2pc/[uurr]|(0.28){}="3"
            \ar@{.}@/_3pc/[uurr]|(0.29){}="4"
        &&&& B \ar[ul]_{v_2}|{}="10"
            \ar@/^2pc/[ulll]^(0.4){v_1}="6"
            \ar@{.}@/^1pc/[uull]|(0.27){}="9"
            \ar@{.}@/^2pc/[uull]|(0.28){}="8"
            \ar@{.}@/^3pc/[uull]|(0.29){}="7"
        \ar@{=>}"1";"2"
        \ar@{=>}"3";"2"
        \ar@{.}"3";"4"^{\alpha_i}
        \ar@{=>}"4";"5"
        \ar@{=>}"6";"7"
        \ar@{.}"8";"7"_{\beta_j}
        \ar@{=>}"8";"9"
        \ar@{=>}"10";"9"
    }}
    \quad\leadsto\quad
    \vcenter{\xymatrix{
        && C \\
        & C_1 \ar[ur]^{w_1}
        && C_2 \ar[ul]_{w_2} \\
        A \ar@{}@<-2pc>[uurr]|(0.28){}="3"
            \ar@{}@<-3pc>[uurr]|(0.29){}="4"
            \ar[ur]^{u_1}
            \ar@/_2pc/[urrr]|(0.61)\hole_(0.4){u_2}
        &&&& B \ar[ul]_{v_2}
            \ar@/^2pc/[ulll]^(0.4){v_1}
            \ar@{}@<2pc>[uull]|(0.28){}="8"
            \ar@{}@<3pc>[uull]|(0.29){}="7"
        \ar@{}"3";"4"^{\simeq \alpha}
        \ar@{}"8";"7"_{\simeq \beta}
    }}
\]

We will now show that equation \eqref{eq:LL} holds for $\alpha$ and $\beta$.
As $\el d F \to \cc{B}$ is a co-fibration, we can lift $\phi_{C'}$ to the co-Cartesian arrow $(\phi_{C'},1_{F\phi_{C'}(z')})$ (as described in Buckley's cleavage introduced in \cite[Prop. 3.3.4]{Buc14}).
Furthermore, as $\el d F \to \cc{B}$ is discrete, there exist unique lifts of $\phi_{w_1'}^{-1}$ and $\phi_{w_2'}^{-2}$, which are necessarily co-Cartesian and invertible, that can be denoted by the same names.
We then have the following diagrams:
\begin{equation} \label{eq:diagsinElF}
    \vcenter{\xymatrix@C=.5pc{
        && (C,F\phi_{C'}(z')) \\
        && (C^{\prime},z') \ar[u]|{(\phi_{C'},1_{F\phi_{C'}(z')})} \\
        & (C_1,z_1) \ar[ur]^{(w_1^{\prime},\kappa)}
            \ar@/^2pc/[uur]^{(w_1,\square_1)}
            \ar@{}[uur]|{\Mreq{\phi_{w_1'}^{-1}}}
        && (C_2,z_2) \ar[ul]_{(w_2^{\prime},\lambda)}
            \ar@/_2pc/[uul]_{(w_2,\square_2)}
            \ar@{}[uul]|{\Mreq{\phi_{w_2'}}} \\
        (A,x) \ar[ur]^{(u_1,\mu_1)}|{}="1"
            \ar@/_2pc/[urrr]_(0.4){(u_2,\mu_2)}="5"
            \ar@{.>}@/_1pc/[uurr]|(0.27){}="2"
            \ar@{.>}@/_2pc/[uurr]|(0.28){}="3"
            \ar@{.>}@/_3pc/[uurr]|(0.29){}="4"
        \ar@{=>}"1";"2"
        \ar@{=>}"3";"2"
        \ar@{.}"3";"4"^{\alpha_i}
        \ar@{=>}"4";"5"
    }}
    \quad
    \vcenter{\xymatrix@C=.5pc{
        & (C,F\phi_{C'}(z')) \\
        & (C^{\prime},z') \ar[u]|{(\phi_{C'},1_{F\phi_{C'}(z')})} \\
        (C_1,z_1) \ar[ur]^{(w_1^{\prime},\kappa)}
            \ar@/^2pc/[uur]^{(w_1,\square_1)}
            \ar@{}[uur]|{\Mreq{\phi_{w_1'}^{-1}}}
        && (C_2,z_2) \ar[ul]_{(w_2^{\prime},\lambda)}
            \ar@/_2pc/[uul]_{(w_2,\square_2)}
            \ar@{}[uul]|{\Mreq{\phi_{w_2'}}} \\
        &&& (B,y) \ar[ul]_{(v_2,\nu_2)}|{}="10"
            \ar@/^2pc/[ulll]^(0.4){(v_1,\nu_1)}="6"
            \ar@{.>}@/^1pc/[uull]|(0.27){}="9"
            \ar@{.>}@/^2pc/[uull]|(0.28){}="8"
            \ar@{.>}@/^3pc/[uull]|(0.29){}="7"
        \ar@{=>}"6";"7"
        \ar@{.}"8";"7"_{\beta_j}
        \ar@{=>}"8";"9"
        \ar@{=>}"10";"9"
    }}
\end{equation}
where $\square_1$ is defined by the lifting of the co-Cartesian 2-cell $\phi_{w_1'}^{-1}$, or more explicitly (with notation as in Remark \ref{rem:variance-duality}),
\[
    Fw_1(z_1)
    \xr{F\phi_{w_1^{\prime}}^{-1}(z_1)}
    F(\phi_{C'}w_1')(z_1)
    \xr{(F^2_{\phi_{c^{\prime}},w_1^{\prime}})_{z_1}}
    F(\phi_{C'})F(w_1')(z_1)
    \xr{F\phi_{c^{\prime}}(\kappa)}
    F\phi_{C^{\prime}}(z')
\]
and similarly $\square_2$ is defined by replacing all appearances of the subindex $1$ by $2$, and $\kappa$ by $\lambda$ in the above.
Now we note that, since $(C,\phi)$ is a pseudo-cocone, in particular it satisfies axiom \ref{PC2} in Definition \ref{def:pseudo-cocone} for each of the 2-cells $\alpha_i$ and $\beta_j$.
This means that, once whiskered by $\phi_{C^{\prime}}$, these 2-cells can be expressed with coherence 2-cells of $\phi$.
Firstly, this implies that all the whiskerings $\phi_{C^{\prime}} \star \alpha_i$ and $\phi_{C^{\prime}} \star \beta_j$ are invertible.
Secondly, this implies that the pastings of the diagrams in \eqref{eq:diagsinElF} can be computed and are the liftings of respectively $\alpha$ and $\beta$ to $\el d F$.

The diagrams in \eqref{eq:diagsinElF}, and the observation above, show that $(\phi_{C^{\prime}}, 1_{F\phi_{C^{\prime}}(z')})(w_1',\kappa)(v_1,\nu_1) \simeq (\phi_{C^{\prime}}, 1_{F\phi_{C^{\prime}}(z')})(k_j,\tau_j)$ for all $j$, and we consider it for the value of $j$ such that $\tau_j$ is invertible (that exists as shown above \eqref{zig-zag-in-C-prime}).
As both arrows on the right hand side of this equality are co-Cartesian arrows, using in turn Propositions \ref{prop:basicpropCartarrows} and \ref{prop:onesided2for3Cartarrows} we conclude that so is $(\phi_{C^{\prime}}, 1_{F\phi_{C^{\prime}}(z')})(w_1',\kappa)$.
By definition of the co-fibration $\el d F$ (see Proposition \ref{prop:grothendieckyieldscofibration}), this means that $F\phi_{C^{\prime}}(\kappa)$ is invertible (since it is an equivalence in a discrete bicategory), and hence so is the arrow $\square_1$ as defined above.

Now, by definition of the 2-cells of $\el d F$ (as in Remark \ref{rem:variance-duality}), the fact that the pastings of the diagrams in \eqref{eq:diagsinElF} are respectively $\alpha$ and $\beta$ means that we have the following two commutative diagrams in $FC$:
\[
    \xymatrix{
        F(w_1u_1)(x) \ar[dd]_{(F\alpha)_x}
            \ar[r]^{\square_1^u}
        & Fw_1(z_1) \ar[dr]^{\square_1} \\
        && F\phi_{C^{\prime}}(z') \\
        F(w_2u_2)(x) \ar[r]_{\square_2^u}
        & Fw_2(z_2) \ar[ur]_{\square_2} 
    }
    \qquad 
    \xymatrix{
        & Fw_1(z_1) \ar[dl]_{\square_1}
        & F(w_1v_1)(y) \ar[dd]^{(F\beta)_y}
            \ar[l]_{\square_1^v} \\
        F\phi_{C^{\prime}}(z') \\
        & Fw_2(z_2) \ar[ul]^{\square_2}
        & F(w_2v_2)(y) \ar[l]^{\square_2^v}
    }
\]
with $\square_1^u$ (resp. with $1 \leftrightarrow 2$, $u\leftrightarrow v$, $x\leftrightarrow y$ and $\mu \leftrightarrow \nu$) defined as the composition
\[
    F(w_1u_1)(x)
    \xr{(F^{-2}_{w_1,u_1})_x}
    F(w_1)F(u_1)(x)
    \xr{Fw_1(\mu_1)}
    Fw_1(z_1).
\]

We can then combine the two commutative diagrams above in the following commutative diagram, where all the arrows are invertible:
\[
    \vcenter{\xymatrix{
        F(w_1u_1)(x) \ar[r]^{(F\alpha)_x}
            \ar[d]_{\square_1^u}
        & F(w_2u_2)(x) \ar[d]^{\square_2^u} \\
        Fw_1(z_1) \ar[r]_{\square_2^{-1} \circ \square_1}
        & Fw_2(z_2) \\
        F(w_1v_1)(y) \ar[u]^{\square_1^v}
            \ar[r]_{(F\beta)_y}
        & F(w_2v_2)(y) \ar[u]_{\square_2^v}
    }}
    \quad \leadsto \quad 
    \vcenter{\xymatrix{
        F(w_1u_1)(x) \ar[r]^{(F\alpha)_x}
            \ar[d]_{\square_1^u}
        & F(w_2u_2)(x) \ar[d]^{\square_2^u} \\
        Fw_1(z_1) \ar[r]_{\square_2^{-1} \circ \square_1}
        & Fw_2(z_2) \\
        F(w_1v_1)(y) \ar@{<-}[u]^{\square_1^{v-1}}
            \ar[r]_{(F\beta)_y}
        & F(w_2v_2)(y) \ar@{<-}[u]_{\square_2^{v-1}}
    }}
\]

Taking $\xi_1 = \nu_1^{-1}\mu_1$ and $\xi_2 = \nu_2^{-1}\mu_2$, the diagram on the right above is precisely the diagram \eqref{eq:LL}.

\noindent
{\bf Construction of $K$.} We now construct an assignement $K\colon \Colim_{\mathcal{B}}F \to  (\pi_0 \el d F)[\cc{W}^{-1}]$ on objects and arrows, which is the identity on objects, and maps a premorphism $(C, u, v, \xi)$ to the roof 
\[\vcenter{\xymatrix@R=.5em{
    & (C,Fv(y))\\
    (A,x) \ar[ur]^(0.3){(u,\xi)}
    && (B,y) \ar[ul]|(0.4){\circ}_(0.3){(v,1_{Fv(y)})}
}},\]
Note that $K$ is well-defined since any homotopy $(C, w_1, w_2, \alpha,\beta)$ as in the proposition yields the roof
\[\xymatrix@R=1.5em@C=0.5em{
    && (C,Fw_2\circ Fv_2(y)) \ar@{<-}[dr]^(0.7){\quad (w_2,1_{Fw_2\circ Fv_2(y)})}|{\circ}
        \ar@{<-}[dl]_{(w_1,\kappa)} \\
    & (C_1,Fv_1(y)) \ar@{<-}[drrr]_{(v_1,1_{Fv_1(y)})}|{\circ}
        \ar@{<-}[dl]_{(u_1,\xi_1)}
    && (C_2,Fv_2(y)) \ar@{<-}[dr]^(0.7){(v_2,1_{Fv_2(y)})}|(0.6){\circ}
        \ar@{<-}[dlll]^{(u_2,\xi_2)}|!{[ll];[dr]}\hole \\
    (A,x)
    &&&& (B,y)
}\]
where $\kappa = F^{-2}_{w_2,v_2}(y) \circ (F\beta)_y \circ F^2_{w_1,v_1}(y)$.

\noindent
{\bf Isomorphism of Categories.} Starting with a premorphism, and applying $K$ and $H$ consecutively, we get the same premorphism we started with. 
In particular, note that two premorphisms are homotopic if and only if the associated roofs are equivalent, so the homotopy relation is indeed an equivalence relation.
Starting instead with a roof and applying $H$ and $K$ consecutively, we get a new roof which is in the same class, furthermore they have not only a common roof but what's called in \cite[Rem. 3.6]{Fri11} an \textit{elementary equivalence}:
\[\xymatrix{
    & (C,z) \ar@/^1pc/[r]^{[(1_c,\nu^{-1} \circ (F^{-0}_c)_z)]}
        \ar@{<-}[drr]_{[(v,\nu)]}|{\circ}
        \ar@{<-}[dl]_{[(u,\mu)]}
    & (C,Fv(y)) \ar@{<-}[dr]^{[(v,1_{Fv(y)})]}|{\circ}
        \ar@{<-}[dll]^{[(u,\nu^{-1}\mu)]} \\
    (A,x)
    &&& (B,y)
}\]
This establishes the existence of  the isomorphism of categories
\[\Colim F \simeq (\pi_0 \el d F)[\cc{W}^{-1}]\]
\end{proof}

\begin{remark}[Composition]\label{rem:oncompositioninthepseudocolim}
The formula in \eqref{eq:compositioninthepseudocolim} for the composition of arrows in the pseudo-colimit is independent of the choices of $f,g$, and $\gamma$.
Indeed, this formula is no other than the one coming from the composition of roofs in the category of fractions, using the assignements $K$ and $H$ constructed in the proof of Proposition \ref{prop:filteredcolimit}.
We make this explicit as follows: given $[(C_1, u_1, v_1, \xi_1)], [(C_2, u_2, v_2, \xi_2)]$ as in \eqref{eq:compositioninthepseudocolim} we apply $K$ and we have two roofs,
\[\vcenter{\xymatrix@C=3pc{
    & (C_1,Fv_1(y))
    && (C_2,Fv_2(z)) \\
    (A,x) \ar[ur]^(0.3){(u_1,\xi_1)}
    && (B,y) \ar[ul]|(0.4){\circ}^(0.3){(v_1,1_{Fv_1(y)})}
        \ar[ur]_(0.3){(u_2,\xi_2)}
    && (C,z) \ar[ul]|(0.4){\circ}_(0.3){(v_2,1_{Fv_2(z)})}
}}\]
and $f,g,$ and $\gamma$ provide the 2-cell $\gamma$ in $\el d F$ in the diagram below, so that the square commutes in $\pi_0 \el d F$ 
\begin{equation}\label{eq:compositionofroofs}\vcenter{\xymatrix@C=3pc{
    && (D,FgFv_2(z)) \ar@{}[dd]|{\gamma} \\
    & (C_1,Fv_1(y)) \ar[ur]^(0.3){(f,\square)}
    && (C_2,Fv_2(z)) \ar[ul]|(0.4){\circ}_(0.3){(g,1_{FgFv_2(z)})} \\
    (A,x) \ar[ur]^(0.3){(u_1,\xi_1)}
    && (B,y) \ar[ul]|(0.4){\circ}^(0.3){(v_1,1_{Fv_1(y)})}
        \ar[ur]_(0.3){(u_2,\xi_2)}
    && (C,z) \ar[ul]|(0.4){\circ}_(0.3){(v_2,1_{Fv_2(z)})}
}}\end{equation}

The arrow $\square$, that exists by the fibration properties of $\el d F$, can be defined as the composition
\[
    FfFv_1(y)
    \xr{(F^{2}_{f,v_1})_y}
    F(fv_1)(y)
    \xr{(F\gamma)_y}
    F(gu_2)(y)
    \xr{(F^{-2}_{g,u_2})_y}
    FgFu_2(y)
    \xr{Fg(\xi_2)}
    FgFv_2(z).
\]
The formula in \eqref{eq:compositioninthepseudocolim} is the one obtained by applying $H$ to the roof in \eqref{eq:compositionofroofs} formed by the four outer arrows.
\end{remark}

\begin{remark}[Elementary Homotopy]\label{rem:elemhomocolimit}
As in Fritz's work \cite[Rem. 3.6]{Fri11}, we have a notion of elementary homotopy here too: a triple $(w,\alpha,\beta)\colon(c_1, u_1, v_1, \xi_1)\equiv (C_2, u_2, v_2, \xi_2)$ where $w\colon C_1 \to C_2$ is an arrow and $\alpha\colon wu_1 \simeq u_2$ and $\beta\colon wv_1 \simeq v_2$ are invertible 2-cells such that
\[\xymatrix{
    Fw\circ Fu_1 (x) \ar[r]^{(F^2_{w, u_1})_x}
        \ar[d]_{Fw(\xi_1)}
    & F(wu_1)(x) \ar[r]^{(F\alpha)_x}
    & F(u_2)(x) \ar[d]^{\xi_2} \\
    Fw\circ Fv_1 (y) \ar[r]_{(F^2_{w, v_1})_y}
    & F(wv_1)(y)\ar[r]_{(F\beta)_y}
    & F(v_2)(y)
}\]
Those elementary homotopies then generate the whole equivalence relation.
In particular, for any premorphism $(C,u,v,\xi)$ and any arrow $w\colon C \to C'$, we have an elementary homotopy generated by $w$:
\[
    (w,1_{wu},1_{wv})\colon
        (C,u,v,\xi)
        \equiv
        (C',wu,wv,\widehat{Fw(\xi)})
\]
where $\widehat{Fw(\xi)} = (F^2_{w,v})_{y} \circ Fw(\xi) \circ (F^{-2}_{w,u})_{x}$.
\end{remark}

\section{Two Basic Properties of Bicategories of Fractions}\label{sec:homcat}
We show in this section the generalization of two classical results from \cite{GZ67} to bicategories: 
\begin{enumerate}
    \item
    In \cite{GZ67} it is shown that each hom-set of the category of fractions can be constructed as a filtered colimit of sets, that is indexed over a slice category.
    We show in Section \ref{subsec:generalizingGZ} that for a bicategory of fractions one has a similar diagram \eqref{eq:thediagramforthehoms} given by a $\Cat$-valued pseudo-functor, whose domain is filtered (see Lemma \ref{lem:strongWslicecofiltered}), and finally that the hom-categories of the bicategory of fractions can be constructed as the pseudo-colimit of this pseudo-functor (see Proposition \ref{prop:homsascolim}). 

    \item
    In Section \ref{subsec:exactness} we generalize to bicategories another basic result from \cite[Chapter I.3]{GZ67}: the fact that the localization by fractions is exact.
\end{enumerate}

We consider in this section a family $\cc{W}$ of arrows of a bicategory $\cc{B}$, containing the identities and closed under composition and invertible 2-cells, whose arrows we denote by $\wto$.

\subsection{Hom-categories of the localization are filtered colimits}\label{subsec:generalizingGZ}
\begin{definition}[Slice $\mathcal{W}/A$]\label{def:Wslice}
For $A\colon\mathcal{B}$, we have a bicategory $\mathcal{W}/A$ defined as the full sub-bicategory of the pseudo-slice bicategory $\mathcal{B} / A$, whose objects are given by the arrows in $\mathcal{W}$.
More explicitly, it has
\begin{enumerate}[align=left, leftmargin=*]
    \item[\textbf{Objects:}]
    $(C, w)$, where $C\colon\mathcal{B}$ and $w\colon C \wto A$ in $\mathcal{W}$.
    
    \item[\textbf{Arrows:}]
    $(f,\alpha)\colon (C_1, w_1) \to (C_2, w_2)$, where $f\colon C_1 \to C_2$ is an arrow of $\mathcal{B}$ and $\alpha\colon w_1 \Rightarrow w_2f$ is an invertible 2-cell.

    \item[\textbf{2-cells:}]
    $\xi\colon (f_1,\alpha_1) \to (f_2,\alpha_2)$ where $\xi\colon f_1 \Rightarrow f_2$ is a 2-cell of $\mathcal{B}$ such that we have the equality of pasting diagrams
    \[\xymatrix{
        C_1 \ar[rr]_{f_1}^{}="1"
            \ar[dr]|{\circ}_{w_1}
            \ar@{}[drr]|{\Mreq{\alpha_1}}
            \ar@/^1.5pc/[rr]^{f_2}_{}="2"
            \ar@{}@<.75pc>[rr]|{\xi \, \Uparrow}
        && C_2 \ar[dl]|{\circ}^{w_2}
        & =
        & C_1 \ar[rr]^{f_2}
            \ar[dr]|{\circ}_{w_1}
            \ar@{}[drr]|{\Mreq{\alpha_2}}
        && C_2 \ar[dl]|{\circ}^{w_2} \\
        & A
        &&&& A
        &
    }\]
\end{enumerate}
$\mathcal{W}/A$ comes equipped with a forgetful pseudo-functor (that is in fact a strict functor) $U\colon \mathcal{W}/A \to \mathcal{B}$.
\end{definition}

\begin{remark}\label{rem:sliceasfibration}
We can also define the lax slice bicategory $\mathcal{B} // A$, just as above, but without asking for the 2-cell ``$\alpha$'' appearing in the arrows to be invertible.
This is a particular case of a lax comma bicategory of a diagram $\mathcal{C} \mr{F} \mathcal{B} \ml{G} \mathcal{D}$, as constructed for example in \cite[4.2.1]{Buc14} (take $F=1_{\mathcal{B}}$, $\mathcal{D}=\mathbf{1}$, and $G=A$).
It follows from  \cite[4.2.5]{Buc14} that $U\colon \mathcal{B}//A \to \mathcal{B}$ is a fibration.
Note that the Cartesian arrows in $\mathcal{B}//A$ are precisely those for which the 2-cell $\alpha$ is invertible.
Also note that, since we assume $\cc{W}$ to be closed:

\smallskip
\noindent {\bf ($\star$)}
    given any Cartesian arrow $(f,\alpha)\colon (C_1, w_1) \to (C_2, w_2)$ of $\mathcal{B}//A$, if $f$ and $w_2$ are in $\cc{W}$, then so is $w_1$ (that is, $(C_1, w_1)$ is an object of $\mathcal{W}/A$).
\end{remark}

\begin{lemma}[Cofiltered Slices]\label{lem:strongWslicecofiltered}
Let $A \colon \mathcal{B}$.
If $\mathcal{W}$ satisfies right fractions, then $\mathcal{W}/A$ is  a cofiltered bicategory as in Definition \ref{def:filtered}.
\end{lemma}
\begin{proof} 
In view of Remark \ref{rem:sliceasfibration}, applying Lemma \ref{lem:1-fibrationsliftfractions} to the fibration $U\colon \mathcal{B}//A \to \mathcal{B}$, we get that the family $\cc{C}_\cc{W}$ of Cartesian arrows over $\cc{W}$ (that is the arrows $(f,\alpha)$ as in Definition \ref{def:Wslice} such that $f$ is in $\cc{W}$ and $\alpha$ is invertible) satisfies fractions in $\mathcal{B}//A$.
We show how this implies that this same family $\cc{C}_\cc{W}$ also satisfies fractions when restricted to $\mathcal{W}/A$ (note that $\mathcal{W}/A$ is locally full in $\mathcal{B}//A$).
Consider axiom \ref{0-Frc} for $\cc{C}_\cc{W}$ in $\mathcal{B}//A$:
\[\xymatrix{
    D \ar@{.>}[r]^h
        \ar@{.>}[d]_u|{\circ}
        \ar@{}[dr]|{\Mreq{\alpha}}
    & A \ar[d]^w|{\circ} \\
    C \ar[r]_f
    & B
}\]
If $A,B,C,$ and $f$ are in $\mathcal{W}/A$, then by the statement marked \textbf{($\star$)} in Remark \ref{rem:sliceasfibration}  so is $D$, and by Propositions \ref{prop:basicpropCartarrows} and \ref{prop:onesided2for3Cartarrows} so is $h$.
It also follows from \textbf{($\star$)} that $\cc{C}_\cc{W}$ satisfies the axioms \ref{1-Frc} and \ref{2-Frc} in $\mathcal{W}/A$.
Finally, since $\mathcal{W}/A$ has $(A,id_A)$ as a (bi)terminal object, and the arrows into it are in $\cc{C}_\cc{W}$, we conclude by Remark \ref{rem:whenfractionsimplycofiltered}.
\end{proof}

A proof of the following result could also be obtained using the properties of the fibration $U\colon \mathcal{B}//A \to \mathcal{B}$, without asking $\gamma$ to be invertible, but we consider a direct explicit proof of this case to be clearer.

\begin{lemma}\label{lem:liftofsquaresforU}
$U\colon \mathcal{W}/A \to \mathcal{B}$ has the following {\em lift of squares} property: given a cospan
\[
    (C,w)
    \mr{(u,\alpha)}
    (C_2,w_2)
    \ml{(v,\beta)}
    (C',w')
\]
an object $D\colon \mathcal{B}$, arrows $h \colon D \to C$, $g\colon D \to C'$ and an invertible 2-cell $\gamma\colon vg \simeq uh$ as on the left, the diagram on the left in $\mathcal{B}$ can be \emph{lifted} to a diagram in $\mathcal{W}/A$:
\[
    \vcenter{\xymatrix{
        & D \ar[ld]_h
            \ar[rd]^g \\
        C \ar[rd]_u
            \ar@{}[rr]|{\Mreq{\gamma}}
        && C' \ar[ld]^v \\
        & C_2
    }}
    \qquad \ml{U} \qquad
    \vcenter{\xymatrix{
        & (D,wh) \ar@{.>}[ld]_{(h,1_{wh})}
            \ar@{.>}[rd]^{(g,\square)} \\
        (C,w) \ar[rd]_{(u,\alpha)}
            \ar@{}[rr]|{\Mreq{\gamma}}
        && (C',w') \ar[ld]^{(v,\beta)} \\
        & (C_2,w_2)
    }}
\]
\end{lemma}
\begin{proof}
We simply note that the lift of $g$, whose second coordinate is denoted by $\square$ in the diagram above, can be (uniquely) defined to make $\gamma$ into a 2-cell of $\mathcal{W}/A$.
Indeed, for this to happen the equation in Definition \ref{def:Wslice} has to be satisfied, that is:
\[
    \vcenter{\xymatrix@R=3pc{
        D \ar@/^1pc/[rr]^{vg}
            \ar@{}[rr]|{\gamma \wr \Downarrow}
            \ar@/_1pc/[rr]|{uh} 
            \ar@<-1ex>[rd]_{wh}
        & \ar@{}[d]|{\Mreq{\alpha \star h}}
        & C_2 \ar@<1ex>[dl]^{w_2} \\
        & A
    }}
    \quad = \quad
    \vcenter{\xymatrix@R=3pc{
        D \ar[r]^g
            \ar[rd]_{wh}
        & C' \ar[d]|{w'}
            \ar@{}@<-1ex>[d]_(0.3){\Mreq{\square}}
            \ar@{}@<1ex>[d]^(0.3){\Mreq{\beta}}
            \ar[r]^v
        & C_2 \ar[dl]^{w_2} \\
        & A
    }}
\]
We observe that $\square$ is uniquely defined as the pasting of the 2-cell on the left and the inverse of $\beta \star g$.
\end{proof}

Let $A,B\colon\mathcal{B}$ and define a pseudo-functor
\begin{equation}\label{eq:thediagramforthehoms}
    F_A^B\colon
        (\mathcal{W}/A)^{\textnormal{op}}
        \mr{U}
        \cc{B}^{\textnormal{op}}
        \xr{\cc{B}(-,B)}
        \Cat
\end{equation}

\begin{proposition}\label{prop:homsascolim}
When $\mathcal{W}/A$ is cofiltered, the pseudo-colimit of $F_A^B$ as constructed in Proposition \ref{prop:filteredcolimit} is a  category that is isomorphic to the following one (that we denote as $\cc{B}[\cc{W}^{-1}](A,B)$, since it matches the one constructed in \cite[\S 2.3]{Pro96} as the hom-categories of the bicategory of fractions $\cc{B}[\cc{W}^{-1}]$):
\begin{enumerate}[align=left, leftmargin=*]
    \item[\textbf{Objects (arrows in $\mathcal{B}[\mathcal{W}^{-1}]$):}]
    Triples $(C, w, f)\colon A \to B$ where $C\colon\mathcal{B}$ is an object, $w\colon C \wto A$ is an arrow of $\mathcal{W}$, and $f\colon C\to B$ is an arrow
    \[\xymatrix{
        & C \ar[dl]_w|{\circ}
            \ar[dr]^f \\
        A
        && B
    }\]

    \item[\textbf{Arrows (2-cells in $\mathcal{B}[\mathcal{W}^{-1}]$):}]
    Equivalence classes of quintuples 
    \[
        (C,u,v, \alpha, \xi)\colon(C_1,\allowbreak w_1, f_1) \to (C_2, w_2, f_2)
    \]
    where $C\colon\mathcal{B}$ is an object, $u\colon C\to C_1$, $v\colon C \to C_2$ are arrows,  $\alpha\colon w_1u \simeq w_2v$ is an invertible 2-cell, and $\xi\colon f_1u \to f_2v$ is a 2-cell such that $w_1u$ (and therefore also $w_2v$) is isomorphic to\footnote{In \cite[\S 2.3]{Pro96}, since $\cc{W}$ is assumed to be closed under invertible 2-cells, the words ``isomorphic to" can be omitted, but note that this result holds without that assumption.} an arrow in $\mathcal{W}$,
    \begin{equation}\label{eq:arrowsinbicatoffrac}\xymatrix{
        && C \ar[dl]_u^{}="1"
            \ar[dr]^v_{}="3" \\
        & C_1 \ar[dl]_{w_1}|{\circ}
            \ar[drrr]|!{[dl];[rr]}\hole^(0.7){f_1}_(0.2){}="4"
        && C_2 \ar[dlll]_(0.7){w_2}|(0.7){\circ}_(0.2){}="2"
            \ar[dr]^{f_2} \\
        A
        &&&& B
        \ar@/^/@{=>}^(0.8){\wr \alpha}"1";"2"
        \ar@/^/@{=>}|(0.55){\mbox{\strut\;}}^(0.2){\xi}"4";"3"
    }\end{equation}
    
    Two such quintuples, $(C,u,v, \alpha, \xi)$, $(\tilde{C},\tilde{u},\tilde{v},\tilde{\alpha}, \tilde{\xi})\colon(C_1, w_1, f_1) \to (C_2, w_2, f_2)$ are equivalent if there exists a homotopy 
    \[
        (\bar{C}, h, \tilde{h}, \gamma, \delta) \colon
            (C,u,v, \alpha, \xi)
            \to
            (\tilde{C},\tilde{u},\tilde{v},\tilde{\alpha}, \tilde{\xi})
    \]
    where $\bar{C}\colon\mathcal{B}$ is an object, $h\colon\bar{C} \to C$, $\tilde{h}\colon\bar{C} \to \tilde{C}$ are arrows, $\gamma\colon uh \simeq \tilde{u}\tilde{h}$ and $\delta\colon vh \simeq \tilde{v}\tilde{h}$ are invertible 2-cells such that $w_1uh$ is isomorphic to an arrow in $\mathcal{W}$ and we have the equalities of pastings of 2-cells:
    \begin{equation}\label{eq:eqbetweenfractions1}\vcenter{\xymatrix{
        & \tilde{C} \ar[ddr]^{\tilde{u}}
            \ar[r]^{\tilde{v}}
        & C_2 \ar[dr]^{w_2}|{\circ}
        &&&& \tilde{C} \ar[r]^{\tilde{v}}
        & C_2 \ar[dr]^{w_2}|{\circ} \\
        \bar{C} \ar[ur]^{\tilde{h}}
            \ar[dr]_h
            \ar@{}[rr]|(0.4){\gamma \wr \Uparrow}
        &&& A \ar@{}[ll]|(0.4){\tilde{\alpha} \wr \Uparrow}
        & =
        & \bar{C} \ar[ur]^{\tilde{h}}
            \ar[dr]_h\ar@{}[rr]|(0.4){\delta \wr \Uparrow}
        &&& A \ar@{}[ll]|(0.4){\alpha \wr \Uparrow} \\
        & C \ar[r]^u
        & C_1 \ar[ur]_{w_1}|{\circ}
        &&&& C \ar[uur]^v
            \ar[r]^u
        & C_1 \ar[ur]_{w_1}|{\circ}
    }}\end{equation}
    \begin{equation}\label{eq:eqbetweenfractions2}\vcenter{\xymatrix{
        & \tilde{C} \ar[ddr]^{\tilde{u}}
            \ar[r]^{\tilde{v}}
        & C_2 \ar[dr]^{f_2}
        &&&& \tilde{C} \ar[r]^{\tilde{v}}
        & C_2 \ar[dr]^{f_2} \\
        \bar{C} \ar[ur]^{\tilde{h}}
            \ar[dr]_h 
            \ar@{}[rr]|(0.4){\gamma \wr \Uparrow}
        &&& B \ar@{}[ll]|(0.4){\tilde{\xi} \Uparrow}
        & =
        & \bar{C} \ar[ur]^{\tilde{h}}
            \ar[dr]_h\ar@{}[rr]|(0.4){\delta \wr \Uparrow}
        &&& B \ar@{}[ll]|(0.4){\xi \Uparrow} \\
        & C \ar[r]^u
        & C_1 \ar[ur]_{f_1}
        &&&& C \ar[uur]^v
            \ar[r]^u
        & C_1 \ar[ur]_{f_1}
    }}\end{equation}
    
    \item[\textbf{Identities:}]
    For $(C,w,f)$ an object, we have the identity $1_{(C,w,f)} = [(C,1_C,1_C,1_{w1_C},1_{f1_C})]$.
    
    \item[\textbf{Composition (vertical composition in $\mathcal{B}[\mathcal{W}^{-1}]$):}]
    As defined in  \cite[p.258]{Pro96}, and recalled in the proof of this proposition.
\end{enumerate}
\end{proposition}
\begin{proof}
Computing the pseudo-colimit of $F_A^B$ using the formula in Proposition \ref{prop:filteredcolimit} gives the following category

\begin{enumerate}[align=left, leftmargin=*]
    \item[\textbf{Objects:}]
    Pairs $((C, w), f)$, which we can consider as triples $(C, w, f)$ as in Proposition \ref{prop:homsascolim}.
    
    \item[\textbf{Premorphisms:}]
    Quadruples $((C,w),(u,\alpha_u),(v,\alpha_v), \xi)\colon(C_1, w_1, f_1) \to (C_2, w_2, f_2)$, with $C$, $u$, $v$ and $\xi$ as in Proposition \ref{prop:homsascolim}, $w\colon C \wto A$ an arrow of $\mathcal{W}$, and $\alpha_u\colon w \simeq w_1u$ and $\alpha_v\colon w \simeq w_2v$ invertible 2-cells.
    
    Note that such a premorphism yields a quintuple as in Proposition \ref{prop:homsascolim} by taking $\alpha = \alpha_v \alpha_u^{-1}$.
    
    \item[\textbf{Homotopies:}]
    Quintuples 
    \[
        ((\bar{C},\bar{w}), (h, \eps), (\tilde{h}, \tilde{\eps}), \gamma, \delta) \colon ((C,w),(u,\alpha_u),(v,\alpha_v), \xi) \equiv ((\tilde{C},\tilde{w}),(\tilde{u},\tilde{\alpha_u}),(\tilde{v},\tilde{\alpha_v}), \tilde{\xi}),
    \]
    with $\bar{C}, h, \tilde{h},\gamma,\delta$ as in Proposition \ref{prop:homsascolim}, $\bar{w}\colon \bar{C} \wto C$ is an arrow in $\mathcal{W}$, $\eps\colon \bar{w} \simeq wh$ and $\tilde{\eps}\colon \bar{w} \simeq \tilde{w}\tilde{h}$ are invertible 2-cells and the following three equations hold (the first two state that $(h, \eps), (\tilde{h}, \tilde{\eps})$ are morphisms in $\mathcal{W}/A$, and the third one is \eqref{eq:LL})
    \[\xymatrix{
        & \tilde{C} \ar@/^/[dr]^{\tilde{u}}
            \ar@{}[d]|{\gamma \wr \Uparrow}
        &&&& \tilde{C} \ar@/^/[dr]^{\tilde{u}}
            \ar[dd]_{\tilde{w}}|{\circ} \\
        \bar{C} \ar@/^/[ur]^{\tilde{h}}
            \ar@/_/[dr]_{\bar{w}}|{\circ}
            \ar[r]^h
        & C \ar@{}[dl]|(0.4){\eps\wr\Uparrow}
            \ar@{}[dr]|(0.4){\alpha_u\wr\Uparrow} 
            \ar[r]^u 
            \ar[d]^w|{\circ}
        & C_1 \ar@/^/[dl]^{w_1}|{\circ}
        & =
        & \bar{C} \ar@/^/[ur]^{\tilde{h}}
            \ar@/_/[dr]_{\bar{w}}|{\circ}
        & \ar@{}[l]|{\tilde{\eps}\wr\Uparrow}
            \ar@{}[r]|{\tilde{\alpha_u}\wr\Uparrow}
        & C_1 \ar@/^/[dl]^{w_1}|{\circ} \\
        & A
        &&&& A
    }\]
    \[\xymatrix{
        & \tilde{C} \ar@/^/[dr]^{\tilde{v}}
            \ar@{}[d]|{\delta \wr \Uparrow}
        &&&& \tilde{C} \ar@/^/[dr]^{\tilde{v}}
            \ar[dd]_{\tilde{w}}|{\circ} \\
        \bar{C} \ar@/^/[ur]^{\tilde{h}}
            \ar@/_/[dr]_{\bar{w}}|{\circ}
            \ar[r]^h
        & C \ar@{}[dl]|(0.4){\eps\wr\Uparrow}
            \ar@{}[dr]|(0.4){\alpha_v\wr\Uparrow}
            \ar[r]^v
            \ar[d]^w|{\circ}
        & C_2 \ar@/^/[dl]^{w_2}|{\circ}
        & =
        & \bar{C} \ar@/^/[ur]^{\tilde{h}}
            \ar@/_/[dr]_{\bar{w}}|{\circ}
        & \ar@{}[l]|{\tilde{\eps}\wr\Uparrow}
            \ar@{}[r]|{\tilde{\alpha_v}\wr\Uparrow}
        & C_2 \ar@/^/[dl]^{w_2}|{\circ} \\
        & A
        &&&& A
    }\]
    \[\xymatrix{
        & \tilde{C} \ar[ddr]^{\tilde{u}}
            \ar[r]^{\tilde{v}}
        & C_2 \ar[dr]^{f_2}
        &&&& \tilde{C} \ar[r]^{\tilde{v}}
        & C_2 \ar[dr]^{f_2} \\
        \bar{C} \ar[ur]^{\tilde{h}}
            \ar[dr]_h \ar@{}[rr]|(0.4){\gamma \wr \Uparrow}
        &&& B \ar@{}[ll]|(0.4){\tilde{\xi} \Uparrow}
        & =
        & \bar{C} \ar[ur]^{\tilde{h}}
            \ar[dr]_h\ar@{}[rr]|(0.4){\delta \wr \Uparrow}
        &&& B \ar@{}[ll]|(0.4){\xi \Uparrow} \\
        & C \ar[r]^u
        & C_1 \ar[ur]_{f_1}
        &&&& C \ar[uur]^v
            \ar[r]^u
        & C_1 \ar[ur]_{f_1}
    }\]
\end{enumerate}

Note that, taking $\alpha = \alpha_v \alpha_u^{-1}$ and $\tilde{\alpha} = \tilde{\alpha}_v \tilde{\alpha_u}^{-1}$, and combining the first two equations, \eqref{eq:eqbetweenfractions1} follows and \eqref{eq:eqbetweenfractions2} is the third equation.
This shows that any homotopy between two premorphisms yields an equivalence between the assigned quintuples, and in this way we have an assignement on objects and arrows $H\colon \Colim{F^A_B} \to \cc{B}[\cc{W}^{-1}](A,B)$ that we will show is in fact an isomorphism of categories. 

First, we note that any quintuple $(C,u,v, \alpha, \xi)$ as in Proposition \ref{prop:homsascolim} defines a premorphism by choosing an arrow $w\colon C \wto A\in\cc{W}$, an isomorphism $\alpha_u\colon w \simeq w_1 u$ and putting $\alpha_v$ as the  composition $w  \Mreq{\alpha_u} w_1 u \Mreq{\alpha} w_2 v$.
And finally, we let the reader check that any quintuple $(\bar{C}, h, \tilde{h}, \gamma, \delta)$ as in Proposition \ref{prop:homsascolim} defines a homotopy between the so-defined premorphisms, by choosing an arrow $\bar{w}\colon \bar{C} \wto A \in \cc{W}$, an isomorphism $\bar{w} \simeq w_1 u h$ and putting
\[
    \eps\colon
        \bar{w}
        \simeq
        w_1 u h
        \Mreq{\alpha_u^{-1} \star h}
        w h,
    \qquad\qquad\qquad
    \tilde{\eps}\colon
        \bar{w}
        \simeq
        w_1 u h
        \Mreq{w_1 \star \delta}
        w_1 \tilde{u} \tilde{h}
        \Mreq{\tilde{\alpha}_u^{-1} \star \tilde{h}}
        \tilde{w} \tilde{h}.
\]
These constructions are clearly mutually inverse, so we have an assignement on objects and arrows $K\colon \cc{B}[\cc{W}^{-1}](A,B) \to  \Colim{F^A_B}$, strictly inverse to $H$.
$K$ maps identities to identities by definition, so we will finish the proof by showing that $K$ preserves the composition of the categories.

We consider thus two composable arrows of $\cc{B}[\cc{W}^{-1}](A,B)$,
\[
    (C_1,w_1,f_1)
    \xrightarrow{[(C,u,v, \alpha, \xi)]}
    (C_2,w_2,f_2)
    \xrightarrow{[(C',u',v', \alpha', \xi')]}
    (C_3,w_3,f_3),
\]
and an invertible 2-cell $\gamma$ fitting as follows:
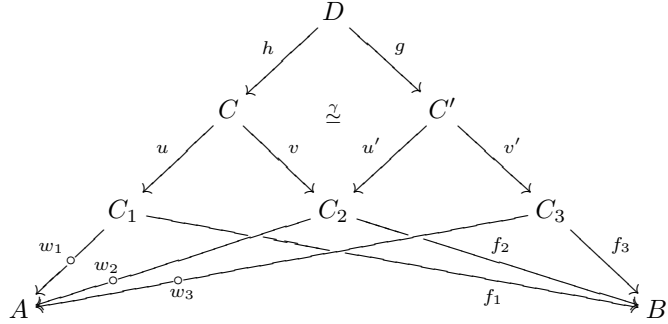
\[\xymatrix{
    &&& D \ar[dr]^{g}
        \ar[dl]_{h}
        \ar@{}[dd]|{\Mreq{\gamma}} \\
    && C\ar[dl]_u
        \ar[dr]^v
    && C' \ar[dl]_{u'}
        \ar[dr]^{v'} \\
    & C_1 \ar[dl]_{w_1}|{\circ}
        \ar[drrrrr]_(0.7){f_1}|!{[dl];[rr]}\hole|!{[dl];[rrrr]}\hole
    && C_2 \ar[dlll]_(0.7){w_2}|(0.7){\circ}
        \ar[drrr]^{f_2}|!{[dlll];[rr]}\hole
    && C_3\ar[dlllll]^(0.7){w_3}|(0.7){\circ}
        \ar[dr]^{f_3} \\
    A
    &&&&&& B
}\]
In the situation in \cite{Pro96}, when defining the vertical composition in the bicategory of fractions, $\gamma$ is a chosen square, and the new quintuple whose class gives the composition is defined by pasting respectively the $\alpha$'s with $\gamma$, and the $\xi$'s with $\gamma$ (see \cite[p.258]{Pro96} for details). 

We can apply $K$ to this new arrow of $\cc{B}[\cc{W}^{-1}](A,B)$, and we claim that we get the same premorphism if we \emph{lift} $\gamma$ to a 2-cell of $\cc{W}/A$ and use this 2-cell to compute the composition between the two induced premorphisms of $\Colim{F^A_B}$, this is just a matter of following these constructions: we first apply $K$ to each of the two arrows above:
\[
    ((C_1,w_1),f_1)
    \xrightarrow{[(C,w)(u,\alpha_u)(v,\alpha \alpha_u), \xi)]}
    ((C_2,w_2),f_2) \xrightarrow{[(C',w')(u',\alpha_{u'})(v',\alpha' \alpha_{u'} ), \xi')]}
    ((C_3,w_3),f_3).
\]
We see then that we can compose these premorphisms of $\Colim{F^A_B}$ as in the item \textbf{Composition} in Proposition \ref{prop:filteredcolimit}, by choosing an invertible 2-cell in $(\cc{W}/A)^{\textnormal{op}}$ of the following form 
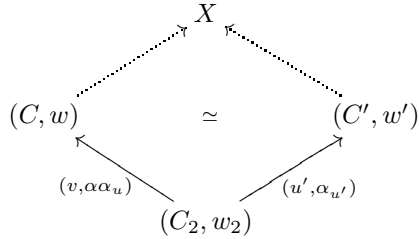
\[\xymatrix{
    & X \ar@{<.}[ld]
        \ar@{<.}[rd] \\
    (C,w) \ar@{<-}[rd]_{(v,\alpha \alpha_u)}
        \ar@{}[rr]|{\simeq}
    && (C',w') \ar@{<-}[ld]^{(u',\alpha_{u'})} \\
    & (C_2,w_2)
}\]
and in fact the composition in the category $\Colim{F^A_B}$ is independent of this choice, as shown in Remark \ref{rem:oncompositioninthepseudocolim}. 
If we choose this 2-cell by lifting $\gamma$ to $\cc{W}/A$ as below, using Lemma \ref{lem:liftofsquaresforU}:
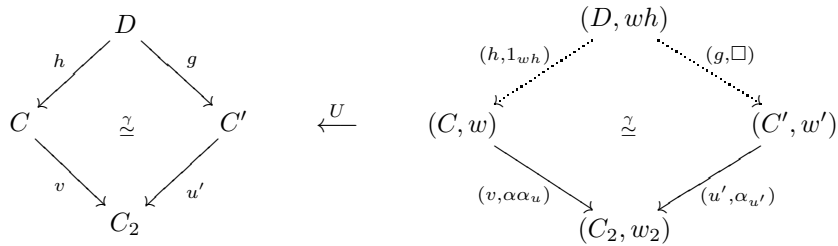
\[
    \vcenter{\xymatrix{
        & D \ar[ld]_h
            \ar[rd]^g \\
        C \ar[rd]_v
            \ar@{}[rr]|{\Mreq{\gamma}}
        && C' \ar[ld]^{u'} \\
        & C_2
    }}
    \qquad \ml{U} \qquad
    \vcenter{\xymatrix{
        & (D,wh) \ar@{.>}[ld]_{(h,1_{wh})}
            \ar@{.>}[rd]^{(g,\square)} \\
        (C,w) \ar[rd]_{(v, \alpha \alpha_u)}
            \ar@{}[rr]|{\Mreq{\gamma}}
        && (C',w') \ar[ld]^{(u',\alpha_{u'})} \\
        & (C_2,w_2)
    }}
\]
then we can compute the composition of the premorphisms by replacing the values in the formula \eqref{eq:compositioninthepseudocolim}.
We let the reader check that applying the assignement $H\colon \Colim{F^A_B} \to \cc{B}[\cc{W}^{-1}](A,B)$ described above to this premorphism 
gives back precisely the formulas in \cite[p.258]{Pro96}.
\end{proof}

\begin{remark}\label{rem:vertcompositioninfractions}
We note that this proof shows, in particular, that the vertical composition of 2-cells in \cite[p.258]{Pro96} does not depend on the choice of the 2-cell $\gamma$ in $\cc{B}$. 
This result is proved directly in \cite[Prop. 5.1]{Tom16}.
Note that Propositions \ref{prop:homsascolim} (the homs of the bicategory of fractions are  filtered pseudo-colimits of categories) and \ref{prop:filteredcolimit} (filtered pseudo-colimits of categories are  1-categories of fractions), combined, state that the homs of the bicategory of fractions  are 1-categories of fractions.
In view of this, and as explained in Remark \ref{rem:oncompositioninthepseudocolim}, we show here that the basic reason why this vertical composition is independent of this choice is that the composition of arrows in the 1-category of fractions is independent of the choice of the commutative square.
\end{remark}    

\subsection{On the exactness of the calculus of fractions}\label{subsec:exactness}
We show now how Proposition \ref{prop:homsascolim} is one of the two \emph{ingredients} needed in order to generalize the exactness of the localization by fractions from categories (\cite[I, Prop 3.1]{GZ67}) to bicategories.
The other required result is the commutativity in $\Cat$ of filtered pseudo-colimits with finite limits.
This has been shown in \cite{Can16}, \cite{DDS18} for the \emph{strict} case, and we generalize it here to the \emph{bicategorical} case.

Let $\mathcal{B}$, $\mathcal{C}$ be bicategories, $F\colon \mathcal{B} \times \mathcal{C} \to \Cat$, $W\colon \mathcal{C} \to \Cat$ be  pseudo-functors.
We denote $W$-weighted bilimits by $\Lim_{W}$ and conical bicolimits, with indexing bicategory $\cc{B}$, by $\Colim_{\mathcal{B}}$. 
There is a canonical functor
\begin{equation}\label{eq:diamondcomparison}
    \Diamond \colon \Colim_{\mathcal{B}} \Lim_{W} F \longrightarrow \Lim_{W} \Colim_{\mathcal{B}} F
\end{equation}
and the ($W$-weighted) bilimit is said to commute with the (conical) bicolimit when this functor is an equivalence of categories (of course, this makes sense for any bicategory $\cc{D}$ in place of $\Cat$, considering equivalences in $\cc{D}$ instead).
In \cite{Can16}, \cite{DDS18}, this is shown to be the case when $\cc{B}$ is filtered, $W$ is a finite weight (which implies in particular that $\cc{C}$ is finite, see \cite[Def. 3.1]{DDS18}), and assuming furthermore when $\cc{B}$, $\cc{C}$, $W$, and $F$ are required to be strict:

\begin{theorem}\label{th:exactnessstrict}
Let $\cc{B}$, $\cc{C}$ be (strict) 2-categories such that $\cc{B}$ is filtered, let $W\colon \mathcal{C} \to \Cat$ be a (strict) 2-functor that is a finite weight, and let $F\colon \mathcal{B} \times \mathcal{C} \to \Cat$ be a (strict) 2-functor.
Then the canonical functor $\Diamond$ in \eqref{eq:diamondcomparison} is an equivalence of categories.
\qed
\end{theorem}

\begin{remark} \label{rem:basicpropofbilimits}
Let $\mathcal{C},\mathcal{D}$ be bicategories and $F\colon \mathcal{C} \to \cc{D}$, $W\colon \mathcal{C} \to \Cat$ be pseudo-functors.
Recall that the $W$-weighted bilimit $\Lim_{W} F$ is defined as a birepresentation of the pseudo-functor (pseudo-presheaf) given by the composition
\[\begin{tikzcd}
	{\textnormal{Cones}_W^F(-) \colon \mathcal{D}^{\textnormal{op}}}
	& {[\mathcal{D}, \Cat]}
	& {[\mathcal{C}, \Cat]}
	&& \Cat
	\arrow["{\yo_{\mathcal{D}^{\textnormal{op}}}}", from=1-1, to=1-2]
	\arrow["{F^*}", from=1-2, to=1-3]
	\arrow["{[\mathcal{C}, \Cat](W,-)}", from=1-3, to=1-5]
\end{tikzcd}\]
where $\yo$ denotes the Yoneda embedding and $F^*$ denotes the pre-composition by $F$.
Note that this pseudofunctor maps $D\colon \mathcal{D}$ to the category of pseudo-natural transformations
\[
    \textnormal{Cones}_W^F(D) = [\cc{C},\Cat](W, \cc{D}(D,F(-)))
\]
from $W$ to $\cc{D}(D,F(-)) \colon \mathcal{C}\to \Cat$, which is the formula that can be found in the original definition of weighted bilimit in \cite[(1.12)]{Str80}.

We recall that, by definition, being birepresentable is a property that is stable under equivalences in $[\cc{D}^{\textnormal{op}}, \Cat]$ (see \cite[(1.11)]{Str80}).
Hence, if another indexing bicategory $\mathcal{C}'$, another weight  $W'\colon \mathcal{C}' \to \Cat$, and another pseudo-functor $F'\colon \mathcal{C}' \to \cc{D}$ induce an equivalent pseudo-presheaf, then the bilimit $\Lim_{W} F$ will coincide with $\Lim_{W'} F'$.
We will use this to show in \textbf{(A)} and \textbf{(B)} how \emph{equivalent} choices of $\mathcal{C}$, $W$ and $F$ lead to equivalent pseudo-presheaves and hence to the same bilimit.
Recall also (see Remark \ref{rem:equivalences}, and for example \cite[1.10]{PW14} for a proof) that equivalences are pointwise in functor bicategories.

\begin{enumerate}[align=left, leftmargin=*, label = \textbf{(\Alph*)}]
    \item
    For each pseudo-functor $S\colon \mathcal{C}' \to \mathcal{C}$, we can consider $S^*\colon [\mathcal{C}, \Cat]\to [\mathcal{C}', \Cat]$, and we have an induced pseudo-natural transformation given by the pasting
    \begin{equation} \label{eq:justapasting}\begin{tikzcd}
    	{\mathcal{D}^{\textnormal{op}}}
    	& {[\mathcal{D}, \Cat]}
    	& {[\mathcal{C}, \Cat]}
    	&& \Cat \\
    	&& {[\mathcal{C}', \Cat]}
    	\arrow["{\yo_{\mathcal{D}^{\textnormal{op}}}}", from=1-1, to=1-2]
    	\arrow["{F^*}", from=1-2, to=1-3]
    	\arrow[""{name=0, anchor=center, inner sep=0}, "{[\mathcal{C}, \Cat](W,-)}", from=1-3, to=1-5]
    	\arrow["{S^*}", from=1-3, to=2-3]
    	\arrow["{(FS)*}"', from=1-2, to=2-3]
    	\arrow[""{name=1, anchor=center, inner sep=0}, "{[\mathcal{C}, \Cat](WS,-)}"', curve={height=12pt}, from=2-3, to=1-5]
    	\arrow["{S^*_{W,-}}"', shorten >=3pt, Rightarrow, from=0, to=1]
    \end{tikzcd}\end{equation}
    where the triangle on the left-hand side is strictly commutative and the 2-cell on the right-hand side is given by the local hom-functors of the pseudofunctor $S^*$,
    \[
        S^*_{W,-} \colon [\mathcal{C}, \Cat](W,-) \to [\mathcal{C}', \Cat](S^*(W),S^*(-)).
    \]
    If $S$ is a biequivalence, then so is $S^*$ and, since biequivalences are locally equivalences and natural pointwise equivalences are equivalences, so is the pseudo-natural transformation in \eqref{eq:justapasting}.
    As we explained above, we then have $\Lim_{W} F= \Lim_{WS} FS$.
    
    \item
    For each pair of pseudo-natural tranformations $\alpha\colon W' \Rightarrow W$ and $\beta\colon F \Rightarrow F'$, we have an induced pseudo-natural transformation given by the pasting
    \begin{equation}\label{eq:justanotherpasting}\begin{tikzcd}
    	{\mathcal{D}^{\textnormal{op}}}
    	& {[\mathcal{D}, \Cat]}
    	& {[\mathcal{C}, \Cat]}
    	&& \Cat
    	\arrow["{\yo_{\mathcal{D}^{\textnormal{op}}}}", from=1-1, to=1-2]
    	\arrow[""{name=0, anchor=center, inner sep=0}, "{F^*}", curve={height=-12pt}, from=1-2, to=1-3]
    	\arrow[""{name=1, anchor=center, inner sep=0}, "{[\mathcal{C}, \Cat](W,-)}", curve={height=-12pt}, from=1-3, to=1-5]
    	\arrow[""{name=2, anchor=center, inner sep=0}, "{[\mathcal{C}, \Cat](W',-)}"', curve={height=12pt}, from=1-3, to=1-5]
    	\arrow[""{name=3, anchor=center, inner sep=0}, "{F'^*}"', curve={height=12pt}, from=1-2, to=1-3]
    	\arrow["{\hat{\alpha}}", shorten <=3pt, shorten >=3pt, Rightarrow, from=1, to=2]
    	\arrow["{\beta^*}", shorten <=3pt, shorten >=3pt, Rightarrow, from=0, to=3]
    \end{tikzcd}\end{equation}
    where $\hat{\alpha}$ denotes the pseudo-natural transformation $\yo_{[\mathcal{C}, \Cat]^{\textnormal{op}}}(\alpha) = [\mathcal{C}, \Cat](\alpha,-)$.
    Similarly to \textbf{(A)}, if $\alpha$ and $\beta$ are equivalences, then so is the pseudo-natural transformation in \eqref{eq:justanotherpasting} and we have $\Lim_{W} F= \Lim_{W'} F$. 
\end{enumerate}
Similarly, one can do the same operations \textbf{(A)} and \textbf{(B)} for (weighted) colimits.
\end{remark}

Using this remark, we can now generalize the commutativity result in Theorem \ref{th:exactnessstrict} to the \emph{bicategorical} setting:

\begin{corollary}\label{coro:exactnessinCat} 
The result in Theorem \ref{th:exactnessstrict} also holds for $\cc{B}$, $\cc{C}$ bicategories, and $W$, $F$ pseudo-functors.
\end{corollary}
\begin{proof}
We first note that, since finite weighted bilimits can be constructed using bicotensors with a finite category, biproducts and biequalizers (this result goes back to \cite{Str80}, see also \cite[p.208]{DDS18a} and \cite[Cor. 6.12]{Can16} for a detailed explanation and proof), it suffices to prove Corollary \ref{coro:exactnessinCat} in these three cases (which is in fact precisely how Theorem \ref{th:exactnessstrict} is proved in \cite{DDS18a} and \cite{Can16}).
Observe that in these three cases $\cc{C}$ and $\cc{W}$ are already strict, so we can assume this for this proof, and we do so in what follows.

We consider the 2-category ${\sf st}(\cc{B})$, the biequivalence ${\sf st}(\cc{B}) \simeq \cc{B}$, and the induced biequivalence ${\sf st}(\cc{B}) \times \cc{C} \simeq \cc{B} \times \cc{C}$. 
Note that, since they are equivalent bicategories, if ${\sf st}(\cc{B})$ is filtered, then so is $\cc{B}$.
Composing $F$ with this pseudo-functor, and applying item \textbf{(A)} in Remark \ref{rem:basicpropofbilimits}, we see that we can assume without loss of generality $\cc{B}$ to be a 2-category.
We can then use the fact (see \cite[4.2]{Pow89}, or the nLab entry on pseudo-functors) that any $\Cat$-valued pseudo-functor with domain a 2-category is equivalent to a 2-functor.
Using item {\bf (B)} in Remark \ref{rem:basicpropofbilimits} (taking $\beta$ to be this equivalence, and $\alpha$ the identity), we can then also assume without loss of generality that $F$ is a 2-functor, as in Theorem \ref{th:exactnessstrict}.
\end{proof}

Combining Corollary \ref{coro:exactnessinCat} with Proposition \ref{prop:homsascolim}, we have:

\begin{theorem}\label{th:exactness}
Let $\mathcal{B}$ be a bicategory and $\mathcal{W}$ be a (right) calculus of fractions.
Then the localization pseudo-functor $L_{\mathcal{W}}: \mathcal{B} \to \mathcal{B}[\mathcal{W}^{-1}]$ commutes with finite weighted bilimits.
In other words, for $\mathcal{C}$ a finite bicategory, $W: \mathcal{C} \to \Cat$ a finite weight, and $F: \mathcal{C} \to \mathcal{B}$ a pseudo-functor, such that the finite weighted bilimit of $F$ by $W$ exists, the finite weighted bilimit of $L_{\mathcal{W}}\circ F$ by $W$ exists and the canonical arrow
\[
    L_{\mathcal{W}}(\Lim_W F) \longrightarrow \Lim_W(L_{\mathcal{W}} \circ F)
\]
is an equivalence in $\mathcal{B}[\mathcal{W}^{-1}]$.
\end{theorem}
\begin{proof}
The proof is formally similar to the one in \cite[I, Prop 3.1]{GZ67}. 
We want to show that for each $A \colon \mathcal{B}$ the canonical functor
\[
    \mathcal{B}[\mathcal{W}^{-1}](A,\Lim_W F)
    \longrightarrow
    \Lim_W \mathcal{B}[\mathcal{W}^{-1}](A,F(-))
\]
is an equivalence of categories.
By Proposition \ref{prop:homsascolim}, this amounts to showing that so is
\[
    \Colim_{(\mathcal{W}/A)^{\textnormal{op}}}\mathcal{B}(U_A(-), \Lim_W F)
    \longrightarrow
    \Lim_W\Colim_{(\mathcal{W}/A)^{\textnormal{op}}}\mathcal{B}(U_A(-),F(-))
\]
or, equivalently, since representables preserve limits \cite[(1.21)]{Str80},
\[
    \Colim_{(\mathcal{W}/A)^{\textnormal{op}}} \Lim_W \mathcal{B}(U_A(-), F(-))
    \longrightarrow
    \Lim_W\Colim_{(\mathcal{W}/A)^{\textnormal{op}}}\mathcal{B}(U_A(-),F(-)),
\]
where $U_A$ denotes the forgetful strict functor from ${\mathcal{W}/A}$. 
This is precisely the content of Corollary \ref{coro:exactnessinCat} (for the pseudo-functor
${\mathcal{W}/A}^{\textnormal{op}} \times \cc{C} \xr{U_A \times F} \cc{B}^{\textnormal{op}} \times \cc{B} \xr{\cc{B}(-,-)} \Cat$).
\end{proof}

\bibliographystyle{alpha}
\bibliography{source}

\end{document}